\documentclass[10pt]{article}

\usepackage{amsmath}
\usepackage{amsthm}
\usepackage{amssymb}
\usepackage{amscd}
\usepackage{enumerate}
\usepackage{mathabx}
\usepackage[plainpages=false,colorlinks,hyperindex,pdfpagemode=None,bookmarksopen,linkcolor=red,citecolor=blue,urlcolor=blue]{hyperref}

\swapnumbers
\numberwithin{equation}{section}
\theoremstyle{plain}
\newtheorem{theorem}[subsubsection]{Theorem}
\newtheorem*{theorem*}{Theorem}

\newtheorem*{claim}{Claim}
\newtheorem{proposition}[subsubsection]{Proposition}
\newtheorem{lemma}[subsubsection]{Lemma}

\newtheorem{assumption}[subsubsection]{Assumption}
\newtheorem{conjecture}[subsubsection]{Conjecture}
\theoremstyle{definition}
\newtheorem{definition}[subsubsection]{Definition}

\theoremstyle{remark}
\newtheorem{remark}[subsubsection]{Remark}
\newtheorem{remarks}[subsubsection]{Remarks}
\newtheorem{example}[subsubsection]{Example}

\renewcommand{\descriptionlabel}[1]%
         {\hspace{\labelsep}\normalfont{#1}}

\newcommand{\bb}{\mathfrak{b}}
\newcommand{\gf}{\mathfrak{g}}
\newcommand{\hh}{\mathfrak{h}}
\newcommand{\kk}{\mathfrak{k}}
\renewcommand{\ss}{\mathfrak{s}}
\newcommand{\zz}{\mathfrak{z}}

\newcommand{\Ak}{\mathbb{A}_k}
\newcommand{\Ad}{\Ak}

\newcommand{\CC}{\mathbb{C}}
\newcommand{\FF}{\mathbb{F}}

\newcommand{\QQ}{\mathbb{Q}}
\newcommand{\RR}{\mathbb{R}}

\newcommand{\ZZ}{\mathbb{Z}}

\newcommand{\diag}{{\operatorname{diag}}}

\newcommand{\Hom}{\operatorname{Hom}}
\newcommand{\End}{\operatorname{End}}
\newcommand{\Aut}{{\operatorname{Aut}}}
\newcommand{\varchi}{\mathcal{X}}
\newcommand{\Gm}{\mathbb{G}_m}
\newcommand{\Ga}{\mathbb{G}_a}

\newcommand{\GL}{\operatorname{GL}}
\newcommand{\Mat}{\operatorname{Mat}}

\newcommand{\PGL}{\operatorname{PGL}}

\newcommand{\SL}{\operatorname{SL}}

\newcommand{\GSp}{{\operatorname{GSp}}}
\newcommand{\Sp}{{\operatorname{Sp}}}

\newcommand{\Gal}{\operatorname{Gal}}
\newcommand{\Fr}{{\operatorname{Fr}}}
\newcommand{\tr}{\operatorname{tr}}

\newcommand{\aff}{{\operatorname{aff}}}
\newcommand{\cusp}{{\operatorname{cusp}}}
\newcommand{\spec}{{\operatorname{spec}\,}}

\newcommand{\Eis}{{\operatorname{Eis}}}
\newcommand{\Bun}{{\operatorname{Bun}}}
\newcommand{\pos}{{\operatorname{pos}}}
\newcommand{\Rep}{{\operatorname{Rep}}}
\newcommand{\ab}{{\operatorname{ab}}}
\newcommand{\RTF}{{\operatorname{RTF}}}
\newcommand{\Res}{{\operatorname{Res}}}
\newcommand{\Sat}{{\operatorname{Sat}}}

\makeatletter
\def\blfootnote{\xdef\@thefnmark{}\@footnotetext} 
\makeatother

\begin{document}

\date{}
\title{Spherical varieties and integral representations \\ of $L$-functions. }
\author{Yiannis Sakellaridis }


\maketitle 

\begin{abstract}
We present a conceptual and uniform interpretation of the methods of integral representations of $L$-functions (period integrals, Rankin-Selberg integrals). This leads to: (i) a way to classify such integrals, based on the classification of certain embeddings of spherical varieties (whenever the latter is available), (ii) a conjecture which would imply a vast generalization of the method, and (iii) an explanation of the phenomenon of ``weight factors'' in a relative trace formula. We also prove results of independent interest, such as the generalized Cartan decomposition for spherical varieties of split groups over $p$-adic fields (following an argument of Gaitsgory and Nadler).
\end{abstract}

\blfootnote{}

\setcounter{tocdepth}{2}
\tableofcontents

\section{Introduction}

\subsection{Goals} The study of automorphic $L$-functions (and their special values at distinguished points, or \emph{$L$-values}) is very central in many areas of present-day number theory, and an incredible variety of methods has been developed in order to understand the properties of these mysterious objects and their deep links with seemingly unrelated arithmetic invariants. Oddly enough, notwithstanding their elegant and very general definition by Langlands in terms of Euler products, virtually all methods for studying them depart from an integral construction of the form: \begin{quote}A suitable automorphic form (considered as a function on the automorphic quotient $[G]:=G(k)\backslash G(\Ad)$), integrated against a suitable distribution on $G(k)\backslash G(\Ad)$, is equal to a certain $L$-value.\end{quote} For ``geometric'' automorphic forms, such an integral can often be expressed as a a pairing between elements in certain homology and cohomology groups, but the essence remains the same. Given the 
importance of such methods, it appears as a paradox that there is no general theory of integral representations of $L$-functions and, in fact, they are often considered as ``accidents''.

In this article I present a uniform interpretation of a large array of such methods, which includes Tate integrals, period integrals and Rankin-Selberg integrals. This interpretation leads to the first systematic classification of such integrals, based on the classification of certain spherical varieties (see sections \ref{secrs} and \ref{secsmoothaffine}). Moreover, it naturally gives rise to a very general conjecture (Conjecture \ref{mainconjecture}), whose proof would lead to a vast extension of the method and would allow us to study many more $L$-functions than are within our reach at this moment. Finally, it explains phenomena which have been observed in the theory of the relative trace formula, in a way that is well-suited to the geometric methods employed in the proof of the fundamental lemma by Ng\^o \cite{Ngo}. In the course of the article we also prove some results which can be of independent interest, including results on the orbits of hyperspecial and congruence subgroups on the $p$-adic points 
of a spherical variety (Theorems \ref{stratification} and \ref{Iwahoritheorem}).

The main idea is based on the well-known principle that a ``multiplicity-freeness'' property usually underlies integral constructions of $L$-functions. For our present purposes, a ``multiplicity-freeness'' property can be taken to mean that a suitable space of functions $\mathcal S(X)$ on a $G(\Ad)$-space $X$ admits at most one, up to constants, morphism into any irreducible admissible representation $\pi$ of $G(\Ad)$. Here $G$ denotes a connected reductive algebraic group over a global field $k$, and $\Ad$ denotes the ring of adeles of $k$. Such spaces arise as the adelic points of spherical varieties. By definition, a spherical variety for $G$ is a normal variety with a $G$-action such that, over the algebraic closure, the Borel subgroup of $G$ has a dense orbit. Let $X$ be an \emph{affine} spherical variety, and denote by $X^+$ the open $G$-orbit on $X$. A second principle behind the main idea is based on ideas around the geometric Langlands program, according to which the correct ``Schwartz space'' $\
mathcal S(X)$ of functions to consider (which are actually functions on $X^+(\Ad)$, not $X(\Ad)$) should be one reflecting the geometry and singularities of $X$. Then, for every cuspidal automorphic representation $\pi$ of $G$ with ``sufficiently positive'' central character, there is a natural pairing $\mathcal P_X: \mathcal S(X(\Ad)) \otimes \pi\to \CC$ . The weak version of our conjecture (\ref{weakconjecture}) asserts that this pairing admits meromorphic continuation to all $\pi$. (A stronger version, \ref{mainconjecture}, states that an ``Eisenstein series'' construction, obtained by summing over the $k$-points of $X$ and integrating against characters of a certain torus acting on $X$, has meromorphic continuation.) Then, assuming the ``multiplicity-freeness'' property, one expects the pairing to be associated to some $L$-value of $\pi$.

If our variety is of the form $H\backslash G$ with $H$ a reductive subgroup of $G$ then from this construction we recover the period integral of automorphic forms over $H(k)\backslash H(\Ad)$ (\S \ref{ssperiods}). More generally, if $X$ is fibered over such a variety and the fibers are (related to) flag varieties, then we can prove meromorphic continuation using the meromorphic continuation of Eisenstein series, and we recover integrals of ``Rankin-Selberg'' type (\S \ref{ssRS}). Thus, we reduce the problem of finding Rankin-Selberg integrals to the problem of classifying affine spherical varieties with a certain geometry. For smooth affine spherical varieties, this geometric problem has been solved by Knop and Van Steirteghem \cite{KnVS}. By inspection of their tables (section \ref{secsmoothaffine}), we recover some of the best-known constructions, such as those of Rankin and Selberg \cite{Ra,Sel}, Godement and Jacquet \cite{GJ}, Bump and Friedberg \cite{BF}, all spherical period integrals, as well as some 
new ones.

We give an example (\S \ref{sstensor}), involving the tensor product $L$-function of $n$ cuspidal representations on $\GL_2$, to support the point of view that the basic object giving rise to an Eulerian integral related to an $L$-function is the spherical variety $X$ and not a geometry related to flag varieties. Finally, we apply these ideas to the relative trace formula (section \ref{secRTF}) to show that certain ``weight factors'' which have appeared in examples of this theory and are often considered an ``anomaly'' can, in fact, be understood using the notion of Schwartz spaces.

\subsection{Background on the methods}
To an automorphic representation $\pi\simeq \otimes_v' \pi_v$ of a reductive group $G$ over a global field $k$, and to an algebraic representation $\rho$ of its Langlands dual group $^L G$, Langlands attached a complex $L$-function $L(\pi,\rho,s)$, defined for $s$ in some right-half plane of the complex plane as the product, over all places $v$, of local factors $L_v(\pi_v,\rho,s)$.\footnote{At ramified places and for most $\rho$, the definition still depends on the local functoriality conjectures.} 

Despite the beauty of its generality, the definition is of little use when attempting to prove analytic properties of $L$-functions, such as their meromorphic continuation and functional equation. Such properties are usually obtained by integration techniques, namely presenting the $L$-function as some integral transform of an element in the space of the given automorphic representation. Such methods in fact predate Langlands by more than a century, but the most definitive construction (as every automorphic $L$-function should be a $\GL_n$ $L$-function) was studied by Godement and Jacquet \cite{GJ} (generalizing Tate's construction for $\GL_1$, \cite{Tate}), 
who proved the analytic continuation and functional equation of $L(\pi,s):=L(\pi,\operatorname{std},s)$, where $\pi$ is an automorphic representation of $G=\GL_n$ and $\operatorname{std}$ is the standard representation of $^L G=\GL_n(\CC)\times \Gal(\bar k/k)$.
 Their method relies on proving the equality:
\begin{equation}\label{GodementJacquet}
  L(\pi,s-\frac{1}{2}(n-1))= \int_{\GL_n(\Ak)} \left<\pi(g)\phi, \tilde\phi\right> \Phi(g) |\det(g)|^s dg
\end{equation}
where $\phi$ is a suitable vector in $\pi$, $\tilde\phi$ a suitable vector in its contragredient and $\Phi$ a suitable function in $\mathcal S(\Mat_n(\Ak))$, the Schwartz space of functions on $\Mat_n(\Ak)$. The main analytic properties of $ L(\pi,\rho,s)$, then, follow from Fourier transform on the Schwartz space and the Poisson summation formula.

Going several decades back in history, Hecke showed that the standard $L$-function of a cuspidal automorphic representation on $\GL_2$ (with, say, trivial central character) has a presentation as a \emph{period integral}, which in adelic language reads:
\begin{equation}\label{hecke} L(\pi,s+\frac{1}{2}) = \int_{k^\times \backslash \Ak^\times} \phi\left(\left(\begin{array}{cc} a & 0 \\ 0 & 1\end{array}\right)\right) |a|^s da \end{equation}
where, again, $\phi$ is a suitable vector in the automorphic representation under consideration. 

Period integrals (by which we mean integrals over the orbit of some subgroup on the automorphic space $G(k)\backslash G(\Ad)$, possibly against a character of that subgroup) have since been studied extensively, although there are still many open conjectures about their relation to $L$-functions (cf., for instance, \cite{II}). Still, they form perhaps the single class of examples where we have a general principle answering the question: How to write down an integral with good analytic properties, which is related to some $L$-function (or $L$-value)? Piatetski-Shapiro discussed this in \cite{PSEuler}, and suggested that the period integral of a cusp form on a group $G$ over a subgroup $H$ (against, perhaps, an analytic family $\delta_s$ of characters of $H$ as in (\ref{hecke})) should always be related to some $L$-value if the subgroup $H$ enjoys a ``multiplicity-one'' property: $\dim \Hom_{H(\Ad)}(\pi, \delta_s)\le 1$ for every irreducible representation $\pi$ of $G(\Ad)$ and (almost) every $s$.

The method of periods usually fails when the subgroup $H$ is non-reductive, the reason being that, typically, the group $H(\Ad)$ has no closed orbits on $G(k)\backslash G(\Ad)$. Therefore there is no a priori reason that the period integral should have nice analytic properties (as the character $\delta_s$ varies), and one can in fact check in examples (see, for instance, \ref{badexample}) that for values of $s$ such that the period integral converges, it does not represent an $L$-function.

In a different vein, Rankin \cite{Ra} and Selberg \cite{Sel} independently discovered an integral representing the tensor product $L$-function of two cuspidal automorphic representations of $\GL_2$. The integral uses as auxilliary data an Eisenstein series on $\GL_2$ and has the following form:
$$ L(\pi_1\times\pi_2, \otimes, s)= \int_{\PGL_2(k)\backslash\PGL_2(\Ak)} \phi_1(g)\phi_2(g) E(g,s) dg $$
with suitable $\phi_1\in\pi_1, \phi_2\in\pi_2$. 

Later, this method was taken up by Jacquet, Piatetski-Shapiro, Shalika, Rallis, Gelbart, Ginzburg, Bump, Friedberg and many others, in order to construct numerous examples of automorphic $L$-functions expressed as integrals of cusp forms against Eisenstein series, with important corollaries for every such expression discovered. Despite the abundance of examples, however, there has not been a systematic understanding of how to produce an integral representing an $L$-function. 

\subsection{Schwartz spaces and $X$-Eisenstein series} While the method of Godement and Jacquet can also be phrased in the language of Rankin-Selberg integrals (see \cite{GPSR}), the fact that no systematic theory of these constructions exists has led many authors to consider them as coincidental and/or to seek direct generalizations of \cite{GJ}, as being a ``more canonical'' construction (cf.\ \cite{BKgamma}). We adopt a different point of view which treats Godement-Jacquet, Rankin-Selberg, and period integrals as parts of the same concept, in fact a concept which should be much more general!

The basic object here is an affine spherical variety $X$ of the group $G$. The reason that such varieties are suitable is that they are related to the ``multiplicity-free'' property discussed above. For instance, in the category of \emph{algebraic} representations, the ring of regular functions $k[X]$ of an affine $G$-variety is multiplicity-free if and only if the variety is spherical. In the $p$-adic setting and for unramified representations, questions of multiplicity were systematically examined in \cite{SaSpc,SaSph}, and of course in special cases such questions have been examined in much greater detail (see, for example, \cite{Pr}).

The main idea is to associate to every affine spherical variety a space of distributions on $G(k)\backslash G(\Ad)$ which should have ``good analytic properties''. For reasons of convenience we set up our formulations in such a way that the analytic problem does not have to do with varying a character of some subgroup $H$ (the isotropy subgroup of a ``generic'' point on $X$), but with varying a cuspidal automorphic representation of $G$. For instance, to the Hecke integral (for $\PGL_2$) we do not associate the variety $\Gm\backslash \PGL_2$, but the variety $X=\PGL_2$ under the $G=\Gm\times \PGL_2$-action. Our distributions (in fact, smooth functions) on $G(k)\backslash G(\Ad)$ come from a ``Schwartz space'' of functions on $X^+(\Ad)$ via a ``theta series'' construction (i.e.\ summation over $k$-points of $X^+$). Here $X^+$ denotes the open $G$-orbit on $X$. The main conjecture \ref{mainconjecture}, then, states that the integral of these ``$X$-theta series'' against central idele class characters (I call 
this integral an $X$-Eisenstein series), originally defined in some domain of convergence, has meromorphic continuation everywhere. Under additional assumptions on $X$ (related to the ``multiplicity-freeness'' property mentioned above), the pairings of $X$-theta series with automorphic forms should be related, in a suitable sense, to automorphic $L$-functions or special values of those.

The geometric Langlands program provides ideas that allow us to speculate on the form of these Schwartz spaces, motivated also by the work of Braverman and Kazhdan \cite{BK, BK2} on the special case that $X$ is the affine closure of $[P,P]\backslash G$, where $P$ is a parabolic subgroup. Let us discuss this work: The prototype here is the case $X^+=U\backslash \SL_2 = \mathbb A^2\smallsetminus\{0\}$ (where $U$ denotes a maximal unipotent subgroup), $X=\mathbb A^2$ (two-dimensional affine space). The Schwartz space is the usual Schwartz space on $X(\Ad)$ which, by definition, is the restricted tensor product $\mathcal S(X(\Ad)):=\otimes_v' (\mathcal S(k_v^2) : \Phi_v^0)$, where for finite places $k_v$ with rings of integers $\mathfrak o_v$ the ``basic vectors'' $\Phi_v^0$ are the characteristic functions of $X(\mathfrak o_v)=\mathfrak o_v^2$. There is a natural meromorphic family of morphisms: $\mathcal S(X(\Ad))\to I_{B(\Ad)}^{G(\Ad)}(\chi)$ (where $I_P^G$ denotes normalized parabolic induction from the 
parabolic $P$, $B$ denotes the Borel subgroup), and for idele class characters $\chi$ the composition with the Eisenstein series morphism: $\Eis_\chi: I_{B(\Ad)}^{G(\Ad)}(\chi) \to C^\infty(G(k)\backslash G(\Ad))$ provides meromorphic sections of Eisenstein series, whose functional equation can be deduced from the Poisson summation formula on $\Ad^2$ -- in particular, the $L$-factors which appear in the functional equation of ``usual'' (or ``constant'') sections are absent here.

This was found to be the case more generally in \cite{BK, BG, BFGM, BK2}: One can construct ``normalized'' sections of Eisenstein series from certain ``Schwartz spaces'' of functions on $[P,P]\backslash G (\Ad)$ (or $U_P\backslash G(\Ad)$, where $U_P$ is the unipotent radical of $P$). These Schwartz spaces should be defined as tensor products over all places, restricted with respect to some ``basic vector''; and the ``basic vector'' should be the function-theoretic analog of the intersection cohomology sheaf of some geometric model for the space $X(\mathfrak o_v)$.  For instance, if $X$ is \emph{smooth} then the intersection cohomology sheaf is constant, which means that  $\Phi_v^0$ is the characteristic function of $X(\mathfrak o_v)$; this explains the distibutions in Tate's thesis, the work of Godement and Jacquet, and the case of period integrals. (In the latter, the characteristic function of $X(\mathfrak o_v)=H\backslash G(\mathfrak o_v)$ is obtained as the ``smoothening'' of the delta function at the 
point $H1\in X$.)

Such geometric models where recently defined by Gaitsgory and Nadler \cite{GN} for \emph{every} affine spherical variety. They provide us with the data necessary to speculate on a generalization of the Rankin-Selberg method. It should be noted, however, that even to define the ``correct'' functions on $X^+(\Ad)$ out of these geometric models one has to rely on certain natural conjectures on them -- therefore the problem of finding an independent or unconditional definition should be considered as part of the steps which need to be taken towards establishing our conjecture.

\subsection{Comments and acknowledgements} Most of the ingredients in the present work are not new. Experts in the Rankin-Selberg method will recognize in our method, to a lesser of greater extent, the heuristics they have been using to find new integrals. The idea that geometric models and intersection cohomology should give rise to the ``correct'' space of functions on the $p$-adic points of a variety comes straight out of the Geometric Langlands program and the work of Braverman and Kazhdan; I have nothing to offer in this direction.

However, the mixture of these ingredients is new and I think that there is enough evidence that it is the correct one. For the first time, a precise criterion is formulated on how to construct a ``Rankin-Selberg'' integral, reducing the problem to a purely geometric one -- classifying certain embeddings of spherical varieties. And evidence shows that there should be a vast generalization which does not depend on such embeddings. I prove no ``hard'' theorems and, in particular, I do not know how to establish the meromorphic continuation of the $X$-Eisenstein series. Hence, I do not know whether I am putting the cart before the horse -- however, as opposed to other conjectures which have appeared in the literature in the past, the distributions defined here are completely geometric and have nothing to do a priori with $L$-functions, which leaves a lot of room for hope. Finally, this point of view proves useful in explaining the phenomenon of ``weight factors'' in the relative trace formula.

This work started in the fall of 2004 during a semester at New York University and was put aside for most of the time since. I am very grateful to Joseph Bernstein, Daniel Bump, Dennis Gaitsgory, David Ginzburg, Herv\'e Jacquet, David Nadler and Akshay Venkatesh for many useful discussions and encouragement. I also thank a referees for many useful comments.


\section{Elements of the theory of spherical varieties}

\subsection{Invariants associated to spherical varieties}  \label{ssinvariants}

A \emph{spherical variety} for a connected reductive group $G$ over a field $k$ is a normal variety $X$ together with a $G$-action, such that over the algebraic closure the Borel subgroup of $G$ has a dense orbit. 

We denote throughout by $k$ a number field and, unless otherwise stated, we make the following assumptions on $G$ and $X$:

\begin{itemize}
 \item $G$ is a split, connected, reductive group,
 \item $X$ is affine.
\end{itemize}

The open $G$-orbit in $X$ will be denoted by $X^+$, and the open $B$-orbit by $\mathring X^+$ (where $B$ is a fixed Borel subgroup of $G$, whose unipotent radical we denote by $U$).\footnote{Notice that this is different from that of \cite{GN}, but compatible with the notation used in \cite{SaSpc,SaSph,SV}.}

The assumption that $G$ is split is certainly very restrictive, but it is enough to demonstrate our point of view, and convenient because of many geometric and representation-theoretic results which have been established in this case. We will discuss \emph{affine} spherical varieties in more detail later, but we just mention here that a common source of examples is when $X^+= H\backslash G$, a quasi-affine homogeneous variety, and $X=\overline{H\backslash G}^\aff = \spec\,\, k[H\backslash G]$, the \emph{affine closure} of $H\backslash G$, cf.\ \S\ref{ssaffine}.

We will be using standard and self-explanatory notation for varieties and algebraic groups; e.g.\ $\mathcal N(H), \mathcal Z(H), H^0$ will be, respectively, the normalizer, center and connected component of a (sub)group $H$, $\bar Y$ will be the closure of a subvariety $Y$, etc. The isotropy group of a point $x$ under a $G$-action will be denoted by $G_x$ and the fiber over $y\in Y$ of a morphism $X\to Y$ by $X_y$. The base change of an $S$-scheme $Y$ with respect to a morphism $T\to S$ will be denoted by $Y_T$, but if $v$ denotes a completion of a number field $k$ and $Y$ is defined over $k$ then we will be denoting by $Y_v$ the \emph{set} $Y(k_v)$.

Let us discuss certain invariants associated to a spherical variety. First of all, for any algebraic group $\Gamma$ we denote by $\varchi(\Gamma)$ its character group, and for any variety $Y$ with an action of $\Gamma$ we denote by $\varchi_\Gamma(Y)$ the group of $\Gamma$-eigencharacters appearing in the action of $\Gamma$ on $k(Y)$. 
If $\Gamma$ is our fixed Borel subgroup $B$, then we will denote $\varchi_B(Y)$ simply by $\varchi(Y)$. The multiplicative group of non-zero eigenfunctions (semiinvariants) for $B$ on $k(Y)$ will be denoted by $k(Y)^{(B)}$.  If $Y$ has a dense $B$-orbit, then we have a short exact sequence: $0\to k^\times \to k(Y)^{(B)} \to \varchi(Y) \to 0$. 

For a finitely generated $\ZZ$-module $M$ we denote by $M^*$ the dual module $\Hom_\ZZ(M,\ZZ)$.
For our spherical variety $X$, we let $\Lambda_X=\varchi(X)^*$ and  $\mathcal Q = \Lambda_X\otimes_\ZZ \QQ$. A $B$-invariant valuation on $k(X)$ which is trivial on $k^\times$ induces by restriction to $k(X)^{(B)}$ an element of $\Lambda_X$. We let $\mathcal V\subset \mathcal Q$ be the cone\footnote{A \emph{cone} in a $\QQ$-vector space is a subset which is closed under addition and under multiplication by $\QQ_{\ge 0}$, its \emph{relative interior} is its interior in the vector subspace that it spans, and a \emph{face} of it is the zero set, in the cone, of a linear functional which is non-negative on the cone -- hence, the whole cone is a face as well.} generated by \emph{$G$-invariant valuations} which are trivial on $k^\times$, cf.\ \cite[Corollary 1.8]{KnLV}. It is known that it is a polyhedral cone, and in fact that it is a fundamental domain for the action of a finite reflection group $W_X$ on $\mathcal Q$. We denote by $\Lambda_X^+$ the intersection $\Lambda_X\cap \mathcal V$. Under the quotient map 
$\varchi(A)^*\otimes \QQ\to \mathcal Q$, $\mathcal V$ contains the image of the \emph{negative} Weyl chamber of $G$ \cite[Corollary 5.3]{KnLV}.

The \emph{associated parabolic} to $X$ is the standard parabolic $P(X):= \{ p\in G | \mathring X^+ \cdot p = \mathring X^+\}$. Make once and for all a choice of a point $x_0\in \mathring X^+(k)$ and let $H$ denote its stabilizer; hence  $X^+=H\backslash G$, and $HB$ is open in $G$. There is the following ``good'' way of choosing a Levi subgroup $L(X)$ of $P(X)$: Pick $f\in k[X]$, considered by restriction as an element of $k[G]^H$, such that the set-theoretic zero locus of $f$ is $X \smallsetminus \mathring X^+$. Its differential $df$ at $1\in G$ defines an element in the coadjoint representation of $G$, and the centralizer $L(X)$ of $df$ is a Levi subgroup of $P(X)$. We fix throughout a maximal torus $A$ in $B \cap L(X)$. We define $A_X$ to be the torus: $L(X)/(L(X)\cap H) = A/ (A\cap H)$; its cocharacter group is $\Lambda_X$. We consider $A_X$ as a subvariety of $\mathring X^+$ via the orbit map on $x_0$.

The finite reflection group $W_X\subset \End(\mathcal Q)$ for which $\mathcal V$ is a fundamental domain is called the \emph{little Weyl group} of $X$. The set of simple roots of $G$ corresponding to $B$ and the maximal torus $A\subset B$ will be denoted by $\Delta$. Consider the (strictly convex) cone negative-dual to $\mathcal V$, i.e.\ the set $\{\chi\in\varchi(X)\otimes \QQ | \left<\chi,v\right>\le 0 \text{ for every } v \in \mathcal V\}$. The generators of the intersections of its extremal rays with $\varchi(X)$ are called the (simple) \emph{spherical roots}\footnote{The work of Gaitsgory-Nadler \cite{GN} and Sakellaridis-Venkatesh \cite{SV} suggests that for representation-theoretic reasons one should slightly modify this definition of spherical roots. However, the lines on which the modified roots lie are still the same, and for the purposes of the present article this is enough.} of $X$ and their set is denoted by $\Delta_X$. They are known to form the set of simple roots of a based root system with 
Weyl group $W_X$. We will denote by $\Delta(X)$ the subset of $\Delta$ consisting of simple roots in $L(X)$, and by $W_{L(X)}\subset W$ the Weyl groups of $L(X)$, resp.\ $G$. There is a canonical way \cite[Theorem 6.5]{KnHC} to identify $W_X$ with a subgroup of $W$, which normalizes and intersects trivially the Weyl group $W_{L(X)}$ of $L(X)$. The data $\varchi(X), W_X, \mathcal V$ are usually easy to compute by finding a point on the open $B$-orbit and using Knop's action of the Borel subgroup on the set of $B$-orbits \cite{KnOrbits}; for a more systematic treatment, see \cite{Lo}.

If $\mathcal V$ is equal to the image of the negative Weyl chamber, then we say that the variety is a \emph{wavefront} spherical variety. (This term is justified by the proof for asymptotics of generalized matrix coefficients in \cite{SV}.) Symmetric varieties, for example, are all wavefront \cite[\S 5]{KnLV}. Also, motivated by the results of \cite{SaSpc}, we will call \emph{geometric multiplicity} of $X$ the cardinality of the generic non-empty fiber of the map: $\varchi(X)/W_X \to \varchi(A)/W$. While none implies the other, it is usually the case that varieties with geometric multiplicity one are wavefront. On the other hand, let us call \emph{arithmetic multiplicity} of $X$ the torsion subgroup of $\varchi(A)/\varchi(X)$. It was shown in \cite{SaSpc} that, if $F$ is a local non-archimedean field then for an irreducible unramified representation $\pi$ of $G(F)$ which is in general position among $X$-distinguished ones (i.e.\ with $\Hom_G(\pi,C^\infty(X(F)))\ne 0$) we have $\dim\Hom_G(\pi,C^\infty(X(F)))
=1$ if and only if both the geometric and arithmetic multiplicity of $X$ are $1$.

The $G$-automorphism group of a homogeneous $G$-variety $X^+=H\backslash G$ is equal to the quotient $\mathcal N(H)/H$. It is known \cite[Lemma 7.17]{Lo} that for $X^+$ spherical the $G$-automorphisms of $X^+$ extend to any affine completion $X$ of $X^+$. Moreover, it is known that $\Aut^G(X)$ is diagonalizable; the cocharacter group of its connected component can be canonically identified (by considering the scalars by which an automorphism acts on rational $B$-eigenfunctions) with $\Lambda_X \cap\mathcal V\cap (-\mathcal V)$. We will be denoting: $\mathcal Z(X):=(\Aut^G(X))^0$. It will be convenient many times to replace the group $G$ by a central extension thereof and then divide by the subgroup of $\mathcal Z(G)^0$ that acts trivially on $X$, so that the map $\mathcal Z(G)^0\to \mathcal Z(X)$ becomes an isomorphism.

\subsection{Spherical embeddings and affine spherical varieties}\label{ssaffine}

We will use the words ``embedding'', ``completion'' or ``compactification'' of a spherical $G$-variety $X$ for a spherical $G$-variety $\bar X$ (not necessarily complete) with an open equivariant embedding: $X\to\bar X$. A spherical embedding is called \emph{simple} if it contains a unique closed $G$-orbit. Spherical embeddings have been classified by Luna and Vust \cite{LV}; our basic reference for this theory will be \cite{KnLV}. We will now recall the main theorem classifying simple spherical embeddings.

For now we assume that $k$ is an algebraically closed field in characteristic zero. However, for Theorem \ref{classemb} below the assumption on the characteristic is unnecessary, and any result that does not involve ``colors'' holds verbatim without the assumption of algebraic closedness when the group $G$ is split. Let $X$ be a spherical variety and let $X^+$ be its open $G$-orbit. The \emph{colors} of $X$ are the closures of the $B$-stable prime divisors of $X^+$; their set will be denoted by $\mathcal D$. For every $B$-stable divisor $D$ in any completion $X$ of $X^+$ we denote by $\rho(D)$ the element of $\mathcal Q$ induced by the valuation defined by $D$. A \emph{strictly convex colored cone} is a pair $(\mathcal C,\mathcal F)$ with $\mathcal C\subset \mathcal Q$, $\mathcal F\subset \mathcal D$ such that: 
\begin{enumerate}
 \item $\mathcal C$ is a strictly (i.e.\ not containing lines) convex cone generated by $\rho(\mathcal F)$ and finitely many elements of $\mathcal V$,
 \item the intersection of $\mathcal V$ with the relative interior of $\mathcal C$ is non-empty,
 \item $0\notin\rho(\mathcal F)$.
\end{enumerate}

If $X$ is a simple embedding of $X^+$ with closed orbit $Y$, we let $\mathcal F(X)$ denote the set of $D\in\mathcal D$ such that $\bar D\supset Y$, and we let $\mathcal C(X)$ denote the cone in $\mathcal Q$ generated by all $\rho(D)$, where $D$ is a $B$-invariant divisor (possibly also $G$-invariant) in $X$ containing $Y$.

\begin{theorem}[{\cite[Theorem 3.1]{KnLV}}]\label{classemb}
 The association $X\to (\mathcal C(X),\mathcal F(X))$ is a bijection between isomorphism classes of simple embeddings of $X^+$ and strictly convex colored cones.
\end{theorem}


Now let us focus on affine and quasi-affine spherical varieties. We recall from \cite[Theorem 6.7]{KnLV}: 
\begin{theorem}\label{affine}
 A spherical variety $X$ is affine if and only if $X$ is simple and there exists a $\chi\in \varchi(X)$ with $\chi|_{\mathcal V}\ge 0$, $\chi|_{\mathcal C(X)}=0$ and $\chi|_{\rho(\mathcal D\smallsetminus \mathcal F(X))}<0$. In particular, $H\backslash G$ is affine if and only if $\mathcal V$ and $\rho(\mathcal D)$ are separated by a hyperplane, while it is quasi-affine if and only if $\rho(\mathcal D)$ does not contain zero and spans a strictly convex cone.
\end{theorem}

Recall \cite[\S 1.1]{BG} that a variety $Y$ over a field $k$ is called \emph{strongly quasi-affine} if the algebra $k[Y]$ of global functions on $Y$ is finitely generated and the natural map $Y\to \spec k[Y]$ is an open embedding. Then the variety $\overline{Y}^\aff:= \spec k[Y]$ is called the \emph{affine closure} of $Y$.
\begin{proposition}
A homogeneous quasi-affine spherical variety $Y=H\backslash G$ is strongly quasi-affine. If $X:=\overline{H\backslash G}^\aff$ then the data $(\mathcal C(X), \mathcal F(X))$ can be described as follows: Consider the cone $\mathcal R\subset \varchi(X)\otimes \QQ$ generated by the set of $\chi\in \varchi(X)$ such that $\chi|_{\mathcal V}\ge 0$, $\chi|_{\rho(\mathcal D)}\le 0$. Choose a point $\chi$ in the relative interior of $\mathcal R$. Then $\mathcal F(X)=\{D\in\mathcal D| \rho(D)(\chi)=0\}$ and $\mathcal C(X)$ is the cone generated by $\mathcal F(X)$.
\end{proposition}

\begin{remark}
 The first statement of the proposition generalizes a result of Hochschild and Mostow \cite{HM} for the variety $U_P\backslash G$, where $U_P$ is the unipotent radical of a parabolic subgroup $P$ of $G$. Indeed, this variety is spherical under the action of $M\times G$, where $M$ is the reductive quotient of $P$.
\end{remark}

\begin{proof}
As a representation of $G$, $k[Y]$ is locally finite and decomposes: 
\begin{equation}\label{decomposition}
 k[Y]=\oplus_\lambda V_\lambda
\end{equation}
 where $V_\lambda$ is the isotypic component corresponding to the representation with highest weight $\lambda$, and the sum is taken over all $\lambda$ with $V_\lambda\ne 0$. Since the variety is spherical, each $V_\lambda$ is isomorphic to one copy of the representation with highest weight $\lambda$. Moreover, the multiplicative monoid of non-zero highest-weight vectors $k[Y]^{(B)}$ is the submonoid of $k(Y)^{(B)}$ (the group of non-zero rational $B$-eigenfunctions) consisting of regular functions. Regular $B$-eigenfunctions are precisely those whose eigencharacter satisfies $\chi|_{\rho(\mathcal D)}\ge 0$; since the set $\mathcal D$ is finite, the monoid of $\lambda$ appearing in the decomposition (\ref{decomposition}) is finitely-generated. Since the multiplication map: $V_\mu\otimes V_\nu$ has image in the sum of $V_\lambda$ with $\lambda\le \mu+\nu$, and composed with the projection: $k[Y]\to V_{\mu+\nu}$ it is surjective, it follows that the sum of the $V_\lambda$, for $\lambda$ in a set of generators 
for the monoid of $\lambda$'s appearing in (\ref{decomposition}), generates $k[Y]$.

The second condition, namely that $Y\to X$ is an open embedding, follows from the assumption that $Y$ is quasi-affine and the homogeneity of $Y$. Hence, $Y$ is strongly quasi-affine.

The affine closure $X$ has the property that for every affine completion $X'$ of $Y$ there is a morphism: $X\to X'$. The description of $(\mathcal C(X),\mathcal F(X))$ now follows from Theorem \ref{affine} above and Theorem 4.1 in \cite{KnLV}, which describes morphisms between spherical embeddings. Notice that the cone $\mathcal C(X)$, as described, will necessarily contain the intersection of $\mathcal V$ with the cone generated by $\rho(\mathcal D)$ in its relative interior, therefore its relative interior will have non-empty intersection with $\mathcal V$.
\end{proof}

Let us now discuss the geometry of affine spherical varieties. The following is a corollary of Luna's slice theorem:

\begin{theorem}[{\cite[III.1.Corollaire 2]{LuSlices}}]\label{Lunacorollary}
If $G$ is a reductive group over an algebraically closed field $k$ in characteristic zero, acting on an affine variety $X$ so that $k[X]^G=k$, then $X$ contains a closed $G$-homogeneous affine subvariety $Y$ such that the embedding $Y\hookrightarrow X$ admits an equivariant splitting: $X\twoheadrightarrow Y$. If $G$ is smooth then the fiber over any (closed) point $y\in Y$ is $G_y$-equivariantly isomorphic to the vector space of a linear representation of $G_y$.
\end{theorem}

Luna's theorem also states that $Y$ is contained in the closure of any $G$-orbit, which is easily seen to be true in the spherical case since affine spherical varieties are simple.  The $G$-automorphism group ``retracts'' $X$ onto $Y$:

\begin{proposition}\label{Taction}\label{structureaffine}
Let $X$ be an affine spherical $G$-variety and let $Y$ be as in the theorem above, considered both as a quotient and as a subvariety of $X$. Let $T$ be the maximal torus in $\Aut^G(X)$ which acts trivially on $Y$. Then the closure of the $T$-orbit of every point on $X$ meets $Y$. Equivalently, $k[X]^T=k[Y]$.
\end{proposition}

\begin{proof} 
This is essentially Corollary 7.9 of \cite{KnMotion}. More precisely, let us assume that $G$ has a fixed point on $X$, i.e.\ $Y$ is a point. (The question is easily reduced to this case, since every $G_y$-automorphism of the fiber of $X\to Y$ over $y$ extends uniquely to a $G$-automorphism of $X$.) The proof of \emph{loc.cit.\ }shows that for a generic point $x\in X$ there is a one-parameter subgroup $H$ of $\Aut^G(X)$ such that $x\cdot H$ contains the fixed point in its closure. Hence $k[X]^T=k$ and therefore $X$ contains a unique closed $T$-orbit.
\end{proof}

Notice that if $G$ has a fixed point on $X$ then we can embed $X$ into a finite sum $V=\oplus_i V_i$ of finite-dimensional representations of $G$, such that the fixed point is the origin in $V$ and there is a subtorus $T$ of $\prod_i \Aut^G(V_i)$ acting on $X$ with the origin as its only closed orbit. (Simply take $V$ to be the dual of a $G$-stable, generating subspace of $k[X]$.)

\subsection{Generalized Cartan decomposition}\label{ssstratification}

 Let $\mathcal K=\CC((t))$, the field of formal Laurent series over $\CC$, and $\mathfrak O=\CC[[t]]$ the ring of formal power series. If $X^+$ is a homogeneous spherical variety over $\CC$, it was proven by Luna and Vust \cite{LV} that:
\begin{theorem}\label{Cstratification}
 $G(\mathfrak O)$-orbits on $X^+(\mathcal K)$ are parametrized by $\Lambda_X^+$, where to ${\check\lambda} \in \Lambda_X^+$ corresponds the orbit through ${\check\lambda}(t)\in A_X(\mathcal K)$.
\end{theorem}

A new proof was given by Gaitsgory and Nadler in \cite{GN}, which can be used to prove the analogous statement over $p$-adic fields. We revisit their argument, adapt it to the $p$-adic case, and extend it to determine the set of $G(\mathfrak o_F)$-orbits on $X(\mathfrak o_F)$, when $G$ and $X$ are affine and defined over a number field and $F$ is a non-archimedean completion (outside of a finite set of places).

\begin{remark}
 In the case of symmetric spaces similar statements on the set of $G(\mathfrak o_F)$-orbits on $X(F)$ and in a more general setting -- without assuming that $G$ is split -- have been proven by Benoist and Oh \cite{BO}, Delorme and S\'echerre \cite{DS}.
\end{remark}

The argument uses compactification results of Brion, Luna and Vust. We first need to recall a few more elements of the theory of spherical varieties. The results below have appeared in the literature for $k$ an algebraically closed field in characteristic zero, but the proofs hold verbatim when $k$ is any field in characteristic zero and the groups in question are split over $k$. (The basic observation being, here, that in all proofs one gets to choose $B$-eigenfunctions in $k(X)$, and since the variety is spherical and the group is split the eigenspaces of $B$ are one-dimensional and defined over $k$, therefore the chosen eigenfunctions are $k$-rational up to $\bar k$-multiple.)

A \emph{toroidal embedding} of $X^+$ is an embedding $X^c$ of $X^+$ in which no color ($B$-stable divisor which is not $G$-stable) contains a $G$-orbit. Theorem \ref{classemb} implies that \emph{simple} toroidal embeddings are classified by strictly convex, finitely generated subcones of $\mathcal V$. Moreover, the simple toroidal embedding $X^c$ obtained from a simple embedding $X$ by taking the cone $\mathcal C(X^c)=\mathcal C(X)\cap \mathcal V$ comes with a proper equivariant morphism: $X^c\to X$ \cite[Theorem 4.1]{KnLV} which is surjective \cite[Lemma 3.2]{KnLV}.

The local structure of a simple toroidal embedding is given by the following theorem of Brion, Luna and Vust:
\begin{theorem}[{\cite[Th\'eor\`eme 3.5]{BLV}}]\label{localstr}
 Let $X^c$ be a simple toroidal embedding of $X^+$ and let $X^c_B$ denote the complement of all colors. Then $X^c_B$ is an open, $P(X)$-stable, affine variety with the following properties:
\begin{enumerate}
 \item $X_B^c$ meets every $G$-orbit.
 \item If we let $Y^c$ be the closure of $A_X$ in $X_B^c$, then the action map $Y^c\times U_{P(X)}\to X^c_B$ is an isomorphism.
\end{enumerate}
\end{theorem}

We emphasize the structure of the affine toric variety $Y^c$: Its cone of regular characters is precisely $\mathcal C(X^c)^\vee:=\{\chi\in \varchi(X)\otimes\QQ| \left<\chi,v\right>\ge 0\text{ for all }v\in \mathcal C(X^c)\}$, in other words: $$Y^c=\spec k[\mathcal C(X^c)^\vee\cap \varchi(X)].$$ By the theory of toric varieties, the theorem also implies that $X^c$ is smooth if and only if the monoid $\mathcal C(X^c)\cap \Lambda_X$ is generated by primitive elements in its ``extremal rays'' (i.e.\ is a free abelian monoid).

Notice that when $\mathcal V$ is strictly convex (equivalently: $\Aut^G(X^+)$ is finite) then $X^+$ admits a canonical toroidal embedding $\bar X$, with $\mathcal C(\bar X)=\mathcal V$, which is complete. This is sometimes called the \emph{wonderful completion} of $X^+$, although often the term ``wonderful'' is reserved for the case that this completion is smooth. If $\mathcal V$ is not strictly convex then $X^+$ still admits a (non-unique) complete toroidal embedding $\bar X$, which is not simple, but as remarked in \cite[8.2.7]{GN} Theorem \ref{localstr} still holds, with $Y^c$ a suitable (non-affine) toric variety containing $A_X$. The fan of $Y^c$ depends on the chosen embedding $\bar X$, but its support is precisely the dual cone of $\mathcal V$ (i.e.\ the set of cocharacters $\lambda$ of $A_X$ such that $\lim_{t\to 0}\lambda(t)\in Y^c$ is equal to $\Lambda_X^+$).

We will use Theorem \ref{localstr} for two toroidal varieties: First, for a complete toroidal embedding $\bar X$ of $X^+$. Secondly, for the variety $\hat X$ obtained from our affine spherical variety $X$ by taking $\mathcal C(\hat X)=\mathcal C(X)\cap\mathcal V$. Before we proceed, we discuss models of these varieties over rings of integers.

\subsubsection{Models over rings of integers}

We start with toric varieties. Let $\mathfrak o$ be an integral domain with fraction field $k$, and let $Y$ be a simple (equivalently, affine) toric variety for a split torus $T$ over $k$. We endow $T$ with its smooth model $\mathcal T=\mathfrak o[\varchi(T)]$ over $\mathfrak o$. Since $Y=\spec k[M]$ for some saturated monoid $M\subset\varchi(T)$, the $\mathfrak o$-scheme $\mathcal Y=\spec \mathfrak o[M]$ is a model for $Y$ over $\mathfrak o$ with an action of $\mathcal T$, and we will call it the \emph{standard model}. The notion easily extends to the case where $Y$ is not necessarily affine, but defined by a fan. If $T$ and $Y$ are defined over a number field $k$ and endowed with compatible models over the $S$-integers $\mathfrak o_S$ for a finite set $S$ of places of $k$, then these models will coincide with the standard models over $\mathfrak o_{S'}$, for some finite $S'\supset S$.

Now we return to the setting where $k$ is a number field, $G$, $X$, $X^+$, $\bar X$, $\hat X$ are as before (over $k$), and let us also fix a point $x_0\in \mathring X^+(k)$. Then we can choose compatible integral models outside of a finite set of places, such that the structure theory of Brion, Luna and Vust continues to hold for these models:
\begin{proposition}\label{places}
 There are a finite set of places $S_0$ of $k$ and compatible flat models $\mathcal G$, $\mathfrak X$ $\mathfrak{\bar X}$ and $\mathfrak{\hat X}$ for $G$, $X$ $\bar X$ and $\hat X$ over the $S_0$-integers $\mathfrak o_{S_0}$ of $k$ such that:
\begin{itemize}
 \item $S_0$ contains all archimedean places;
 \item the chosen point $x_0 \in  \mathcal{\mathring X^+}(\mathfrak o_{S_0})$;
 \item $\mathcal  G$ is reductive over $\mathfrak o_{S_0}$, $\mathcal {X^+}\to \spec\mathfrak o_{S_0}$ is smooth and surjective; 
 \item the statement of Theorem \ref{localstr} holds for $\mathfrak{\bar X}$ and $\mathfrak{\hat X}$ over $\mathfrak o_{S_0}$: namely, if we denote any one of them by $\mathfrak X^c$ then there is an open, $\mathcal P(X)$-stable subscheme $\mathfrak X^c_B$ and a toric $\mathcal A$-subscheme $\mathcal Y^c$ \emph{of standard type} such that the subscheme $\mathfrak X^c_B$ meets every $\mathcal G$-orbit on $\mathfrak X^c$ and the action map: $\mathcal Y^c\times \mathcal U_{P(X)}\to \mathfrak X^c_B$ is an isomorphism of $\mathfrak o_{S_0}$-schemes.
 \item $\mathfrak{\bar X}$ is proper over $\mathfrak o_{S_0}$, and the morphism $\mathfrak{\hat X}\to \mathfrak X$ is proper.
\end{itemize}
\end{proposition}

\begin{remarks}\begin{enumerate}\item
By $\mathfrak{X^+}$ (resp.\ $\mathfrak{\mathring X^+}$) we denote the complement of the closure, in any of the above schemes, of the complement of $X^+$ (resp.\ $\mathring X^+$) in the generic fiber. 
\item It is implicitly part of the ``compatibility'' of the models that the scheme structures on $\mathfrak{X^+}, \mathfrak{\mathring X^+}$ do not depend on which of the ambient schemes we choose to define them. 
\item We understand the statement ``meets every orbit'' as follows: Let $|\mathcal Z|$ denote the set of scheme-theoretic points of a scheme $\mathcal Z$. Consider the two maps: $p:\mathcal G\times \mathfrak X\to \mathfrak X$ (projection to the second factor) and $a: \mathcal G\times \mathfrak X\to \mathfrak X$ (action map). Then for every $x\in |\mathfrak X^c|$ the set $a(p^{-1}\{x\})$ intersects $|\mathfrak X^c_B|$ non-trivially.\end{enumerate}
\end{remarks}

\begin{proof}
  For a finite set $S$ of places and a flat model $\mathfrak X^c$ of $X^c$ over $\mathfrak o_S$ (assumed proper if $X^c=\bar X$), let $D$ denote the union of all colors over the generic point of $\spec \mathfrak o_S$, let $\mathfrak D$ denote the closure of $D$ in $\mathfrak X^c$ and let $\mathfrak X_B^c$ be the complement of $\mathfrak D$ in $\mathfrak X^c$. Let $\mathcal G$ denote a compatible reductive model for $G$ over $\mathfrak o_S$. (All these choices are possible by sufficiently enlarging $S$.) The image of $\mathcal G\times \mathfrak X^c_B\to \mathfrak X^c$ is open and contains the generic fiber, hence by enlarging the set $S$, if necessary, we can make it surjective.

 Now define $\mathcal Y^c$ as the closure of $Y^c$ in $\mathfrak X^c_B$. By enlarging the set $S$, if necessary, we may assume that $\mathcal Y^c$ is of standard type. The action map $\mathcal Y^c\times \mathcal U_{P(X)}\to \mathfrak X_B^c$ being an isomorphism over the generic fiber, it is an isomorphism over $\mathfrak o_S$ by enlarging $S$, if necessary.
\end{proof}

From now on we fix such a finite set of places $S_0$ and such models. The combinatorial invariants of the above schemes are the same at all places of $S_0$:

\begin{proposition}\label{samedata}
 Each of the data\footnote{Since $\bar X$ is not necessarily simple, it is not described by a cone but by a fan. However, we slightly abuse the common notation here and write $\mathcal C(\bar X)$ for the set of invariant valuations whose center is in $\bar X$ -- i.e.\ for the support of the fan associated to $\bar X$.} $\varchi(X), \mathcal V, \mathcal C(X), \mathcal C(\bar X), \mathcal C(\hat X)$ is the same for the reductions of $\mathfrak X, \mathfrak{\bar X}, \mathfrak{\hat X}$  at all closed points of $\mathfrak o_{S_0}$. The set of $G$-orbits on each of these varieties is in natural bijection with the set of $\mathcal G$-orbits on each of their reductions.
\end{proposition}

\begin{proof}
The toric scheme $\mathcal Y^c$ being of the standard type, it means that $\varchi(X)=\varchi_A(Y^c)$ is the same at all reductions. For every place $v$ of $\mathfrak o_S$ the reductions $\mathfrak{\bar X}_{\FF_v}$, $\mathfrak{\hat X}_{\FF_v}$ are toroidal: Indeed, denoting by $\mathfrak X^c$ either of them, the complement of $(\mathfrak X_B^c)_{\FF_v}$ is a $\mathcal B_{\FF_v}$-stable union of divisors which does not contain any $\mathcal G_{\FF_v}$-orbit, since $(\mathfrak X_B^c)_{\FF_v}$ meets every $\mathcal G_{\FF_v}$-orbit. Moreover, $\mathfrak X_B^c$ meets no colors: for if it did, then a non-open $\mathcal A_{\FF_v}$-orbit on $\mathcal Y^c_{\FF_v}$ would belong to the open $\mathcal G_{\FF_v}$-orbit, and hence the open $\mathcal G_{\FF_v}$-orbit would belong to the closure of a non-open $G$-orbit over the generic point, a contradiction since by assumption $\mathfrak X^+$ is smooth and surjective. Therefore, the complement of $(\mathfrak X_B^c)_{\FF_v}$ is the union of all colors of $\mathfrak X^c_{\
FF_v}$, and $\mathfrak X^c_{\FF_v}$ is toroidal. Moreover, the $\mathcal G_{\FF_v}$-invariant valuations on ${\FF_v}(\mathfrak{X^+}_{\FF_v})$ whose center is in $\mathfrak X^c_{\FF_v}$ are precisely those of $\Lambda_X\cap \mathcal C(X^c)$ (which proves the equality of $\mathcal C(\mathfrak X^c_{\FF_v})$ with $\mathcal C(X^c)$ at all $v\notin S_0$), and from the fact that $\mathfrak{\bar X}_{\FF_v}$ is complete and $\mathcal C(\mathfrak{\bar X}_{\FF_v})=\Lambda_X^+$ it follows that $\mathcal V$ is precisely the cone of invariant valuations on $\FF_v(\mathfrak{X^+})$.
\end{proof}

Now we are ready to apply the argument of \cite[Theorem 8.2.9]{GN} to describe representatives for the set of $\mathcal G(\mathfrak o_F)$-orbits on $\mathfrak{X^+}(\mathfrak o_F)$, for every completion $F$ of $k$ outside of $S_0$, and also extend it to a description of the set of orbits which are contained in $\mathfrak X(\mathfrak o_F)$. Notice that since $\mathcal G$ is reductive,  $\mathcal G(\mathfrak o_F)$ is a hyperspecial maximal compact subgroup of $G(F)$. From now on we denote our fixed models over $\mathfrak o_{S_0}$ by regular script, since there will be no possibility of confusion. There is a canonical $A_X(\mathfrak o_F)$-invariant homomorphism: $A_X(F)\to \Lambda_X$ (under which an element of the form $\lambda(\varpi)$, where $\varpi$ is a uniformizer for $F$, maps to $\lambda$) and we denote by $A_X(F)^+$ the preimage of $\Lambda_X^+$.

\begin{theorem}\label{stratification}
 For $F$ a completion of $k$ outside of $S_0$ each $G(\mathfrak o_F)$-orbit on $X^+(F)$ contains an element of $A_X(F)^+$, and elements of $A_X(F)^+$ with different image in $\Lambda_X^+$ belong to distinct $G(\mathfrak o_F)$-orbits. If the quotient $\varchi(A)/\varchi(X)$ is torsion-free then the map from $G(\mathfrak o_F)$-orbits on $X^+(F)$ to $\Lambda_X^+$ is a bijection. The orbits contained in $X(\mathfrak o_F)$ are precisely those mapping to $\Lambda_X^+\cap \mathcal C(X)$.
\end{theorem}

\begin{remark}
 The torsion of the quotient $\varchi(A)/\varchi(X)$ is the ``arithmetic multiplicity'' defined in \S \ref{ssinvariants}. Is is trivial if and only if the map: $A_X(F)/A(\mathfrak o) \to \Lambda_X$  is bijective, hence the statement about bijectivity in that case is straightforward. In general, elements in different $A(\mathfrak o_F)$-orbits may belong to the same $G(\mathfrak o_F)$-orbit, for instance, if $X^+=H\backslash G$ with $H$ connected then the map: $G(\mathfrak o_F)\ni g\mapsto x_0\cdot g\in X^+(\mathfrak o_F)$ will be surjective by an application of Lang's theorem (the vanishing of Galois cohomology of $H$ over a finite field). But it is also not always the case that elements corresponding to the same $\lambda$ will always be in the same $G(\mathfrak o_F)$-orbit -- for instance, when $H$ is not connected.
\end{remark}

We will prove this theorem together with a theorem about orbits of the first congruence subgroup, which will not be used here but will be useful elsewhere. Let $\mathbb F$ denote the residue field of $F$.

\begin{theorem}\label{Iwahoritheorem}
 Let $K_1, A_{X,1}, U_1$ be the preimages of $1\in G(\FF)$, $1\in A_X(\FF)$, $1\in U(\FF)$ in $G(\mathfrak o_F)$, $A_X(\mathfrak o_F)$, $U(\mathfrak o_F)$, respectively. Then for every $x \in A_X(F)^+$ we have $x\cdot K_1\subset x\cdot A_{X,1}\cdot U_1$.
\end{theorem}

\begin{proof}[Proof of Theorems \ref{stratification} and \ref{Iwahoritheorem}]
Denote $\mathfrak o_F$ by $\mathfrak o$. We use the notation $X^c, X^c_B, Y^c,$ etc.\ as above for the scheme $\bar X$. The $\mathfrak o$-scheme $X^c$ is proper and hence $X^c(\mathfrak o)=X^c(F)$. We will first show that $Y^c(\mathfrak o)$ contains representatives for all $G(\mathfrak o)$-orbits on $X^c(\mathfrak o)$. Let $x\in X^c(\mathfrak o)$ and denote by $\bar x\in X^c(\FF)$ its reduction. The open, $P(X)$-stable subvariety $X_B^c$ meets every $G$-orbit; for a spherical variety for a split reductive group over an arbitrary field (denoted $\FF$, since we will apply it to this field) the $\FF$-points of the open $B$-orbit meet every $G(\FF)$-orbit. (This is proven following the argument of \cite[Lemma 3.7.3]{SaSpc}, i.e.\ reducing to the case of rank one groups, and by inspection of the spherical varieties for $\SL_2$, classified in \cite[Theorem 5.1]{KnR1}.) This means that there is a $\bar g\in G(\FF)$ (which we can lift to a $g\in G(\mathfrak o)$) such that $\overline{x \cdot g}\in X_B^c(\FF)$. Since 
$X_B^c$ is open, this means that $x\cdot g\in X_B^c(\mathfrak o)=Y^c(\mathfrak o)\times U_{P(X)}(\mathfrak o)$. Acting by a suitable element of $U_{P(X)}(\mathfrak o)$, we get a representative for the $G(\mathfrak o)$-orbit of $x$ in $Y^c(\mathfrak o)$.
Hence, $G(\mathfrak o)$-orbits on $X^+(F)$ are represented by elements of $A_X(F)^+=Y^c(\mathfrak o)\cap A_X(F)$. 

To prove that elements mapping to distinct $\lambda, \lambda'\in \Lambda_X^+$ belong to different  $G(\mathfrak o)$-orbits, the argument of Gaitsgory and Nadler carries over verbatim: If $\lambda$ and $\lambda'$ are not $\QQ$-multiples of each other, we can construct as in \cite{KnLV} a toroidal embedding $X^t$ of $X^+$ over $\mathfrak o$ such that $\lambda(\varpi)\in X^t(\mathfrak o)$ but $\lambda'(\varpi)\notin X^t(\mathfrak o)$. Finally, if $\lambda$ and $\lambda'$ are $\QQ$-multiples of each other (without loss of generality: $\lambda\ne 0$), then we can find a toroidal compactification $X^t$ such that $\lim_{t\to 0} \lambda(t)$ belongs to some $G$-orbit $D$ of codimension one, and then the intersection numbers of $\lambda(\varpi)$ and $\lambda'(\varpi)$ (considered as $1$-dimensional subschemes of $X^t$) with $D$ are different. (Notice that the constructions of \cite{KnLV} are over a field of arbitrary characteristic, and based on Proposition \ref{samedata} one can carry them over over the ring $\
mathfrak o_F$.)

To finish the proof of Theorem \ref{stratification}, if we now consider $\hat X$ then we have a proper morphism: $\hat X\to X$ which is an isomorphism on $X^+$. By the valuative criterion for properness, every point in $X(\mathfrak o)\cap X^+(F)$ lifts to a point on $\hat X(\mathfrak o)$, therefore for the last statement it suffices to determine the set of $G(\mathfrak o)$-orbits on $\hat X(\mathfrak o)\cap X^+(F)$. By the same argument as before, every $G(\mathfrak o)$-orbit meets $\hat Y(\mathfrak o)$, and the latter intersects $A_X(F)$ precisely in the union of $A_X(\mathfrak o)$-orbits represented by $\Lambda_X\cap \mathcal C(X)$.

For Theorem \ref{Iwahoritheorem}, we first notice that $X_B^c(\mathfrak o)$ (where $X^c$ still denotes $\bar X$) is $K_1$-stable; indeed, for any $x\in X_B^c(\mathfrak o)$ and $g\in K_1$ the reduction of $x\cdot g$ belongs to $X_B^c(\FF)$, and since $X_B^c$ is open this implies that $x\cdot g\in X_B^c(\mathfrak o)$. Now we claim that $Y^c(\mathfrak o)\cdot U_1$ is also $K_1$-stable; indeed, this is the preimage in $X_B^c(\FF)$ of $Y^c(\FF)$, and  for every $x\in Y^c(\mathfrak o)\cdot U_1, g\in K_1$ the reduction of $x\cdot g$ belongs to $Y^c(\FF)$. We have already argued that elements of $A_X(F)^+$ with different images in $\Lambda_X^+$ belong to distinct $G(\mathfrak o)$-orbits, hence to distinct $K_1$-orbits; hence, $x\cdot K_1$ belongs to the set of elements of $A_X(F)^+ \cdot U_1$ with the same image $\lambda_x\in \Lambda_X^+$ as $x$.

To distinguish between those elements, we assign to them some invariants which will be preserved by the $K_1$-action. First of all, if $\lambda_x=0$ then the reduction of $x$ modulo $\mathfrak p$ is an element of $X^+(\FF)$ which is preserved by $K_1$, and the elements of $A_X(F)^+ \cdot U_1$ having the same reduction are precisely the elements in the same $A_{X,1}\cdot U_1$-orbit as $x$. Assume now that $\lambda_x\ne 0$ and fix as above a spherical embedding $X^t$ of $X^+$ over $\mathfrak o$ such that $\lim_{t\to 0}\lambda(t)$ belongs to a $G$-orbit of codimension one, whose closure we denote by $D$. Let $n$ be the intersection number of $x\in X^t(\mathfrak o)\cap X^+(F)$ with $D$, then $x: \spec \mathfrak o \to X^t$ has reductions $\bar x: \spec \FF\to D$, $\bar x^n: \spec (\mathfrak o/\mathfrak p^n) \to D$ and $\bar x^{n+1}: \spec (\mathfrak o/\mathfrak p^{n+1}) \to X^t$, which give rise to an $\FF$-linear map from the fiber at $\bar x$ of the conormal bundle of $D$ in $X^t$ to $\mathfrak p^n/\mathfrak p^{
n+1}$. The group $K_1$ preserves the reduction of $x$ and acts trivially on the fiber of the conormal bundle of $D$ over it, therefore preserves this map. It is straightforward to see that for elements of $A_X(F)^+ \cdot U_1$ with the same image in $\Lambda_X^+$ this invariant characterizes the $A_{X,1}\cdot U_1$-orbit of $x$.

\end{proof}


\section{Speculation on Schwartz spaces and automorphic distributions}\label{secSchwartz}

This section is highly conjectural and only aims at fixing ideas. We speculate on the existence of some ``Schwartz space'' of functions on the points of an affine spherical variety over a local field, and explain how to construct from it distributions on the automorphic quotient $[G]:=G(k)\backslash G(\Ad)$ which should have good analytic properties. At almost every place this space of functions should come equipped with a distinguished, unramified element which should be related (in a rather ad hoc way, using the generalized Cartan decomposition) to intersection cohomology sheaves on spaces defined by Gaitsgory and Nadler. In subsequent sections we will specialize to the case where $X$ has a certain geometry (which we call a ``pre-flag bundle''), and these distinguished functions will be described explicitly, in order to understand the Rankin-Selberg method.

\subsection{Formalism of Schwartz spaces and theta series}\label{ssformalism}

\subsubsection{Schwartz space} \label{Schwartz} We fix an affine spherical variety $X$ for a (split) reductive group $G$ over a global field $k$, and for every place $v$ of $k$ we denote by $X_v^+$ the space of $k_v$-points of $X^+$. We assume as given, for every $v$, a $G_v$-invariant ``Schwartz space'' of functions $\mathcal S(X_v)\subset C^\infty(X_v^+)$, and for almost every (finite) $v$ a distinguished unramified element $\Phi_v^0\in \mathcal S(X_v)^{G(\mathfrak o_v)}$ (called ``basic vector'' or ``basic function'') such that: 
\begin{equation}\label{oneonsmooth}\Phi_v^0|_{X^+(\mathfrak o_v)} = 1.
\end{equation}
 (Clearly, the integral model which is implicit in the definitions will not play any role.) We also assume the following regarding the support of Schwartz functions and their growth close to the complement of $X^+$:

\begin{itemize}
\item The closure in $X_v$ of the support of any element of $\mathcal S(X_v)$ is compact.
\item There exist a finite set $\{f_1,\dots,f_n\}$ of elements of $k[X]$, whose common zeroes lie in $X\smallsetminus X^+$, and a natural number $n$, such that for any place $v$ and any $\Phi_v\in \mathcal S(X_v)$ there is a constant $c_v$, equal to $1$ for $\Phi_v=\Phi_v^0$, such that for all $x\in X^+(k_v)$ we have:  $|\Phi_v(x)|\le c_v\cdot (\max_i|f_i(x)|)^{-1}$.
\end{itemize}
At archimedean places the requirement of compact support is far from ideal, but for our present purposes it is enough. One should normally impose similar growth conditions on the derivatives (at archimedean places) of elements of the Schwartz space, but we will not need them here.

The corresponding \emph{global Schwartz space} is, by definition, the restricted tensor product:
\begin{equation} \mathcal S(X(\Ad)):= \bigotimes'_v \mathcal S(X_v)
\end{equation}
with respect to the basic vectors $\Phi_v^0$.

Despite the notation, the elements of $\mathcal S(X(\Ad))$ cannot be interpreted as functions on $X(\Ad)$. They \emph{can} be considered, though, as functions on $X^+(\Ad)$, because of the requirement (\ref{oneonsmooth}).

We may require, without serious loss of generality, that $X^+(\Ad)$ carries a positive $G(\Ad)$-eigenmeasure $dx$ whose eigencharacter $\psi$ is the absolute value of an algebraic character. We normalize the regular representation of $G(\Ad)$ on functions on $X^+(\Ad)$ so that it is unitary when restricted to $L^2(X)=L^2(X,dx)$: $$g\cdot \Phi (x):= \sqrt{\psi(g)} \Phi(x\cdot g).$$

The \emph{$X$-theta series} is the following functional on $\mathcal S(X(\Ad))$:
\begin{equation}\theta(\Phi):= \sum_{\gamma\in X^+(k)} \Phi(\gamma).
\end{equation}
Translating by $G(\Ad)$, we can also consider it as a morphism:
\begin{equation}
 \mathcal S(X(\Ad))\to C^\infty([G]),
\end{equation}
which will be denoted by the same letter, i.e.:
\begin{equation}\label{PsES}
\theta(\Phi,g)=\sum_{\gamma\in X^+(k)} (g\cdot\Phi)(\gamma).
\end{equation}

This sum is absolutely convergent, by the first growth assumption. (Notice that $X$ is affine and hence $X(k)$ is discrete in $X(\Ad)$.)

\subsubsection{Mellin transform} \label{sssMellin} Now recall (Proposition \ref{structureaffine}) that, unless $X$ is affine homogeneous, it has a positive-dimensional group of $G$-automorphisms, i.e.\ $\mathcal Z(X)\ne 0$. By enlarging $G$ and dividing by the subgroup of $\mathcal Z(G)^0$ that acts trivially, we will from now on assume that $\mathcal Z(G)^0\simeq \mathcal Z(X)$ under its action on $X$. An algebraic character of $\mathcal Z(X)$ will be called \emph{$X$-positive} if it extends to the closure of a generic orbit of $\mathcal Z(X)$, that is: $\chi: \mathcal Z(X)\to \Gm$ is positive if for $Y = \overline{\mathcal Z(X)\cdot x}$, where $x$ is a generic point (say, a point on the open $G$-orbit) the function $z\cdot x\mapsto \chi(z)\in \Gm\subset\Ga$ extends to a morphism: $Y\to \Ga$. Obviously, $X$-positive characters span a polyhedral cone in $\varchi(\mathcal Z(X))\otimes \QQ$, and we will use the expression ``sufficiently $X$-positive characters'' to refer to characters in the translate of this 
cone by an element belonging to its relative interior. This notion will also be used for complex-valued characters: a sufficiently $X$-positive character is one whose absolute value can be written as the product of the absolute values sufficiently $X$-positive algebraic characters, raised to powers $\ge 1$. Similar notions will be used for the dual cone, in the space of cocharacters into $\mathcal Z(X)$; for example, a cocharacter $\check\lambda$ is $X$-positive if and only if for a generic point $x\in X$ we have $\lim_{t\to 0}x\cdot \check\lambda(t)\in X$. Finally, since by our assumption $\varchi(G)\otimes\QQ=\varchi(\mathcal Z(X))\otimes\QQ$, we can use the notion of $X$-positive characters for characters of $G$, as well.

\begin{proposition}\label{moderategrowth}
 The function $\theta(\Phi,g)$ on $G(k)\backslash G(\Ad)$ is of moderate growth. Moreover, it is compactly supported in the direction of $X$-positive cocharacters into $\mathcal Z(G)$; that is, for every $g\in G(\Ad)$ we have: $$\theta(\Phi, g\cdot \check\lambda(a))=0$$ if $\check \lambda$ is a non-trivial $X$-positive cocharacter into $\mathcal Z(X)=\mathcal Z(G)^0$ and the norm of $a\in \Ad^\times$ is sufficiently large.
\end{proposition}

The statement about the support is an obvious corollary of the compact support of $\Phi$, and the statement on moderate growth will be proven in the next subsections. Assuming it for now, we may consider the \emph{Mellin transform} of $\theta(\Phi,g)$ with respect to the action of $\mathcal Z(G)$:
\begin{equation}\label{ES}
E(\Phi,\omega,g) = \int_{\mathcal Z(X)(\Ad)} \theta(z\cdot \Phi, g) \omega(z) dz,
\end{equation}
originally defined for sufficiently $X$-positive idele class characters $\omega$. We will call this an \emph{$X$-Eisenstein series}.

We have:
\begin{proposition}\label{convergence}
For sufficiently $X$-positive $\omega$, the integral (\ref{ES}) converges and the function $E(\Phi,\omega,g)$ is of moderate growth in $g$.
\end{proposition}

\begin{proof}
 The statement about convergence follows immediately from Proposition \ref{moderategrowth}; the statement on moderate growth is proven in the same way as Proposition \ref{moderategrowth}, and we will not comment on it separately.
\end{proof}

\subsubsection{Adelic distance functions.}\label{sssdistance}

Let $Z\subset X$ be a closed subvariety of an affine variety, and let $X^+$ denote the complement of $Z$. We would like to define some ``natural'' notion of distance  from $Z$ (denoted $d_Z$) for the adelic points of $X^+$. The distance function will be an Euler product:
$$ d_Z(x)=\prod_v d_{Z,v}(x_v)$$
where, for $x\in X^+(\Ad)$, almost all factors will be equal to one.

We do it in the following way: first, we fix a finite set $S$ of places, including the archimedean ones, and an affine flat model for $X$ over the $S$-integers $\mathfrak o_S$. The closure of $Z$ in this model defines an ideal $J\subset\mathfrak o_S[X]$. We can choose a finitely-generated $\mathfrak o_S$-submodule $M$ of $J$ such that $M$ generates $J$ as an $\mathfrak o_S[X]$-module. In the case when $X$ carries the action of a group $G$ and $Z$ is $G$-stable, we also choose a compatible flat model for $G$ over $\mathfrak o_S$ and require that $M$ be $G$-stable (i.e.\ the action map maps $M\to M\otimes_{\mathfrak o_S} \mathfrak o_S[G]$).

Finally, let $\{f_i\}_i$ be a finite set of generators of $M$ over $\mathfrak o_S$. Then for a point $x\in X^+(\Ad)$ we define:
\begin{equation}
 d_{Z,v}(x_v) = \max_i\{|f_i(x_v)|_v\}
\end{equation}
and 
\begin{equation} 
 d_Z(x)=\prod_v d_{Z,v}(x_v).
\end{equation}

We will call this an \emph{adelic distance function} from $Z$. Notice that almost all factors of this product are $1$ since $x\in X^+(\Ad)$. Moreover, the function extends by zero to a continuous function on $X(\Ad)$.

\begin{remark}
For $v\notin S$ the local factor $d_{Z,v}$ depends only on $M$ and not the choices of $f_i$'s: it is the absolute value of the fractional ideal generated by the image of $M$ under $x_v: \mathfrak o_S[X]\to \mathfrak o_v$. Moreover, the restriction of $d_{Z,v}$ to $X(\mathfrak o_v)$ does not depend on $M$, either, since the image of $J$ generates the same fractional ideal. (The restriction of $d_{Z,v}$ to $X(\mathfrak o_v)$ is a height function, i.e.\ $q_v$ raised to the intersection number of $x\in X(\mathfrak o_v)$ with $Z$.)

Finally, the restriction of $d_Z$ to any compact subset of $X(\Ad)$ is up to a constant multiple independent of choices. Indeed, such a compact subset is the product of $X(\mathfrak o_v)$, for $v$ outside of a finite number of places $S'\supset S$, with a compact subset of $\prod_{v\in S'} X(k_v)$, therefore it suffices to prove independence for the $d_{Z,v}$'s when $v\in S'$. For any two sets of functions $\{f_j\}_j, \{f_i'\}_i$ as above we can write $f_i'=\sum_j h_{ij} f_j$ with $h_{ij}\in\mathfrak o_{S}[X]$ and for each $v\in S'$ there is a constant $C_v$ such that $|h_{ij}(x_v)|_v\le C_v$ when $x$ is in the given compact set. Then $\max_i |f_i'(x_v)|_v\le C_v \max_j |f_j'(x_v)|_v$, and therefore $d_Z'(x)\le C d_Z(x)$ in the given compact set, where $C=\prod_{v\in S'} C_v$.
\end{remark}

For two complex valued functions $f_1$ and $f_2$ we will write $f_1\ll^p f_2$ (where the exponent $p$ stands for ``polynomially'') if there exists a polynomial $P$  such that $|f_1|\le P(|f_2|)$. We will say that $f_1$ and $f_2$ are \emph{polynomially equivalent} if $f_1\ll^p f_2$ and $f_2\ll^p f_1$. 

In this language, it is easy to see that the assumption of \S \ref{Schwartz} on growth of Schwartz functions close to the complement of $X^+$ is equivalent to the following: If $Z$ denotes the complement of $X^+$ in $X$ then for any adelic distance function $d_Z$ from $Z$ and any $\Phi\in\mathcal S(X(\Ad))$ we have:
\begin{equation}\label{growth}|\Phi(x)| \ll^p d_Z(x)^{-1}\end{equation}
for every $\Phi\in \mathcal S(X(\Ad))$. 

Indeed, let the functions $f_i$ be as in the assumption of \S \ref{Schwartz} and let the functions $f_j'$ define an adelic distance function as above. By enlarging $S$ we may assume that $f_i\in \mathfrak o_S[X]$ for all $i$, and by enlarging it further we may assume that the support of $\Phi$ is the product of $\prod_{v\notin S} X(\mathfrak o_v)$ with a compact subset of $\prod_{v\in S} X(k_v)$. By the assumption, the functions $f_i$ generate an ideal whose radical contains $J$. Therefore, $(f_i)_i\supset J^n$ for some $J$ and hence for each $j$ there are $h_{ij}\in\mathfrak o_S[X]$ such that:
$$(f_j')^n=\sum_i h_{ij} f_i$$
Therefore for $v\notin S$ and $x_v\in X(\mathfrak o_v)$ we have:
$$d_{Z,v} (x_v)^n \le \max_i |f_i(x)|,$$
and for $v\in S$ we can find $C_v$ such that $|h_{ij}(x_v)|_v\le C_v$ if $x$ is in the support of $\Phi$. Therefore, for $x$ in the support of $\Phi$ we have:
$$\prod_v (\max_i|f_i(x_v)|_v)^{-1} \le \prod_{v\in S} C_v^{-1} \cdot d_Z(x)^{-n}.$$
Vice versa, if $\Phi$ is known to be polynomially bounded by  $d_Z(x)^{-1}$ then it is bounded by a constant times $d_Z(x)^{-n}$ for some $n$ (since $d_Z(x)$ is bounded in the support of $\Phi$), which implies the bound of the assumption.
e

\subsubsection{Proof of Proposition \ref{moderategrowth}.}

Recall that an automorphic function $\phi$ is ``of moderate growth'' if $\phi\ll^p \Vert g\Vert$ for some natural norm $\Vert \bullet \Vert$ on $G_\infty$. Recall that a ``natural norm''  is a positive function on $G_\infty$ which is polynomially equivalent to $\Vert \rho(g)\Vert$ where: $\rho$ denotes an algebraic embedding $G\hookrightarrow \GL_n$, and $\Vert g\Vert:= \max\{\vert g\vert_{l^\infty}, \vert g^{-1}\vert_{l^\infty}\}$ on $\GL_n(k_\infty)$ (where $\vert\bullet\vert_{l^\infty}$ denotes the operator norm for the standard representation of $\GL_n$ on $l^\infty(\{1,\dots,n\})$).

Assume without loss of generality that $\Phi=\otimes_v \Phi_v$, with $\Phi_v\in \mathcal S(X_v)$, and let $S_\Phi=\prod S_{\Phi_v}$ where $S_{\Phi_v}$ is the support of $\Phi_v$ in $X(k_v)$ (a compact subset).

The claim of the Proposition will follow from (\ref{growth}) if, in addition, we establish that (for $g\in G_\infty$ and $x\in X^+(\Ad)$):
\begin{itemize}
 \item $\#(X^+(k)\cap S_\Phi g)\ll^p \Vert g\Vert$.
 \item $\left(\inf d_Z(X^+(k)g)\right)^{-1} \ll^p \Vert g \Vert$.
\end{itemize}

Indeed, assuming these properties we have: $$\theta(\Phi,g)=\sum_{\gamma\in X^+(k)} (g\cdot\Phi)(\gamma) \le \#(X^+(k)\cap S_\Phi g^{-1}) \cdot \sup_{x\in X^+(k)} |\Phi(xg)| \ll^p$$ $$\ll^p \Vert g\Vert \cdot \left(\inf_{x\in X^+(k)} d_Z(xg)\right)^{-1} \ll^p \Vert g\Vert \cdot \Vert g \Vert.$$

The first property is standard, and follows from the analogous claim for $\GL_n$ (after fixing an equivariant embedding of $X$ in the vector space of a representation of $G$), since $S_\Phi$ is a compact subset of $X(\Ad)$.

To prove the second property, we may assume that the elements $f_i\in k[X]$ defining $d_Z$ span over $k$ a $G$-invariant space $M\subset k[X]$ and that the norm on $G_\infty$ is induced by the $l^\infty(\{f_i\}_i)$-operator norm on $\GL(M_\infty)$. (If the homomorphism $G\to \GL(M)$ is not injective, then this $l^\infty$ norm is bounded by some natural norm on $G_\infty$, which is enough for the proof of this property.) Then for every $x\in X_\infty$ and $g\in G_\infty$ we have: $$\Vert g\Vert^{-1} \cdot d_{Z,\infty}(x)  \le d_{Z,\infty}( x\cdot g) \le \Vert g\Vert \cdot d_{Z,\infty}(x)$$
(where we keep assuming that $d_Z$ is defined by a basis for $M$).

We apply this to points $x\in X^+(F)$. Notice that for every $x\in X^+(k)$ we have $f_i(x)\in k$ and $\ne 0$ for at least one $i$, hence $d_Z(x)=\prod_v \max_i |f_i(x)|_v\ge \max_i \prod_v |f_i(x)|_v = 1$. Therefore, we have: $d_{Z,\infty}( x\cdot g) \ge \Vert g\Vert^{-1} \cdot d_{Z,\infty}(x) \ge \Vert g\Vert^{-1}$.

\qed

\subsection{Conjectural properties of the Schwartz space}\label{ssconjectures}

We saw in Proposition \ref{convergence} that, under very mild assumptions on the basic functions $\Phi_v^0$, the Mellin transform of the corresponding $X$-theta series converges for sufficiently $X$-positive characters $\omega$. However, there is no reason to expect in general that it admits meromorphic continuation to the set of all $\omega$. Indeed, this often fails for the most naive choice of basic functions, namely the characteristic functions of $X^+(\mathfrak o_v)$. We discuss an example, which will be encountered again in \S \ref{sstensor}:

\begin{example}\label{badexample} Let $G=(\PGL_2)^3\times \Gm$, and let $H$ denote the subgroup:
$$\left\{ \left.\left(\begin{array}{cc} a & x_1\\ & 1\end{array}\right) \times \left(\begin{array}{cc} a & x_2\\ & 1\end{array}\right) \times \left(\begin{array}{cc} a & x_3\\ & 1\end{array}\right) \times a\right| x_1+ x_2 + x_3 = 0\right\}.$$

If we defined the local Schwartz space to be equal to $C_c^\infty(H_v\backslash G_v)$, with basic function $\Phi_v=$the characteristic function of $(H\backslash G)(\mathfrak o_v)$ (which is equal to the characteristic function of a single $G(\mathfrak o_v)$-orbit), then, as we will explain in more detail in \S \ref{ssperiods}, the integral of a cusp form agains an $X$-Eisenstein series is equal to the period integral of a cusp form on $G$ over $H(k)\backslash H(\Ad)$, and the usual ``unfolding'' method shows that this can be written as:

$$\int_{\Ad^\times} W_1 \left(\begin{array}{cc} a\\ &1\end{array}\right) W_2 \left(\begin{array}{cc} a\\ &1\end{array}\right) W_3 \left(\begin{array}{cc} a\\ &1\end{array}\right) |a|^s da,$$
where the $W_i$'s are Whittaker functions of cusp forms on $\PGL_2$ and the parameter $s$ depends on the restriction of the given representation to $\Gm(\Ad)$ (assumed to factor through the absolute value map, for simplicity). For $\Re(s)$ large this integral can be written as a convergent Euler product of the analogous local integrals.

An explicit but lengthy computation shows that, if the $W_i(1)$ are normalized to be equal to $1$, if $a,b,c$ denote the Satake parameters of the three $\PGL_2$-cusp forms (considered as elements in $\mathbb C^\times$, well defined up to inverse), and if we set $Q=q^{-\frac{3}{2}-s}$ then the local unramified factors of this Euler product are equal, for a certain choice of measure on $\Ad^\times$, equal to:
$$ \frac{(-1+3Q^2+3Q^4-Q^6)+ (Q^2+Q^4)(a^2+a^{-2}+b^2+b^{-2}+c^2+c^{-2})}{\prod_{\sigma=(\sigma_1,\sigma_2,\sigma_3)\in \{\pm 1\}^3} (1-Q a^{\sigma_1} b^{\sigma_2} c^{\sigma_3})}$$
$$ - \frac{2Q^3(a+a^{-1})(b+b^{-1})(c+c^{-1})}{\prod_{\sigma=(\sigma_1,\sigma_2,\sigma_3)\in \{\pm 1\}^3} (1-Q a^{\sigma_1} b^{\sigma_2} c^{\sigma_3})} $$

The denominator of this expression is very pleasant (it is equal to the denominator of the tensor product $L$-function of the three cuspidal representations), but the numerator does not represent an $L$-function and it would be unreasonable to expect that its Euler product admits meromorphic continuation. Therefore, this was not the correct Schwartz space.
\end{example}

The conjectures that follow are very speculative, but will provide the suitable ground for unifying various methods of integral representations of $L$-functions. There are several reasonable assumptions that one could impose on the spherical variety, the strongest of which would be that for every irreducible admissible representation $\pi$ of $G(\Ad)$ we have: $\dim_{G(\Ad)}(\pi, C^\infty(X^+(\Ad)))\le 1$. At the very minimum, we require from now on that the arithmetic multiplicity (\S \ref{ssinvariants}) of $X$ is trivial. Equivalently, at every place $v$ there is a unique open $B(k_v)$-orbit, and this also implies that generic $G$-stabilizers are connected\footnote{If $H$ is not connected then we have a finite cover: $H^0\backslash G\to H\backslash G$ which gives rise to a finite cover of the associated open $B$-orbits. But this implies that the $B$-stabilizer $B_x$ of a generic point is not connected, hence $H^1(k, B_x)\ne 0$, and therefore $\left(B_x\backslash B\right)(k) \supsetneq B_x(k)\backslash B(k)
$.} and therefore, at almost every (finite) place $v$, the space $X^+(\mathfrak o_v)$ is homogeneous under $G(\mathfrak o_v)$.

\begin{conjecture} \label{mainconjecture}
 Given an affine spherical variety $X$ over $k$ with trivial arithmetic multiplicity, there exists a Schwartz space $\mathcal S(X(\Ad))$, in the sense described above, such that:
\begin{itemize}
 \item The basic functions $\Phi_v^0$ factor through the map of the generalized Cartan decomposition: 
$$\{G(\mathfrak o_v)\mbox{-orbits on }X^+_v\}\to \Lambda_X^+$$
and as functions on $\Lambda_X^+$ are equal to the functions obtained via the function-sheaf correspondence from the ``basic sheaf'' of Gaitsgory and Nadler, as will be explained in \ref{sssbasicfunction}.
 \item For every $\Phi\in \mathcal S(X(\Ad))$, the $X$-Eisenstein series $E (\Phi, \omega, g)$, originally defined for sufficiently $X$-positive characters, admits a meromorphic continuation everywhere.
\end{itemize}
\end{conjecture}

\begin{remarks}\label{afterconjecture}
\begin{enumerate}
\item The first property could be taken as the definition of the basic function, if one knew that the functions obtained from the Gaitsgory-Nadler sheaf are independent of some choices, which we will explain in \S \ref{sssbasicfunction}. In any case, such a definition would be very ad hoc and not useful; one hopes that there exists an alternative construction of the Schwartz space, as in \cite{BK}.
\item The property of meromorphic continuation is mostly dependent on the basic vectors and not on the whole Schwartz space; for instance, at a finite number of places we may replace any function with a function whose (loceal) Mellin transform is a meromorphic multiple of the Mellin transform of the original function without affecting the meromorphicity property. Therefore, the properties do not determine the Schwartz space uniquely; they should hold, for instance, if we take $\mathcal S(X_v)$ to be the $G$-space generated by the basic vector and $C_c^\infty(X_v^+)$.
\item The fact that the theta series is defined with reference to the group $G$ (since we are summing over the $k$-points of its open orbit) certainly seems unnatural; it would be more ``geometric'' to sum over the $k$-points of the open subvariety where $\mathcal Z(X)$ acts faithfully. However, this does not affect the validity of Conjecture \ref{mainconjecture}, since one case can be inferred from the other by induction on the dimension of $X$.
\end{enumerate} 
\end{remarks}

The conjecture about meromorphic continuation of the Mellin transform is a very strong one (cf.\ \S \ref{sstensor} for an example) and, in fact, is not even known in the case of usual Eisenstein series, i.e.\ the case of $X=\overline{U_P\backslash G}^\aff$, where $U_P$ is the unipotent radical of a parabolic $P$ (except when $P$ is a Borel subgroup). We now formulate a weaker conjecture which says that the $X$-Eisenstein series can be continued meromorphically ``as functionals on the space of automorphic forms''. In fact, the precise interpretation of them as functionals on the whole space of automorphic forms would require a theory similar to the spectral decomposition of the relative trace formula, which lies beyond the scope of the present paper. Therefore, we confine ourselves to the cuspidal component of this functional. (Notice, however, that there are a lot of interesting examples which have zero cuspidal contribution, e.g.\ $X=\Sp_{2n}\backslash \GL_{2n}$.)

\begin{conjecture}[Weak form]\label{weakconjecture}
Same assumptions as in Conjecture \ref{mainconjecture}, but the second property is replaced by the following:
\begin{itemize}
 \item  For every cusp form $\phi$ on $G(k)\backslash G(\Ad)$, the integral:
\begin{equation}\label{intconj1}
\int_{[G]} \phi\cdot\omega(g) \theta(\Phi,g) dg
 \end{equation}
originally defined for sufficiently $X$-positive idele class characters $\omega$ of $G$, admits meromorphic continuation to the space of all idele class characters of $G$.
\end{itemize}
\end{conjecture}

\begin{remark} Following up on the third remark of \ref{afterconjecture}, we will see in Proposition \ref{parind} that for the large class of \emph{wavefront} spherical varieties (see \S \ref{ssinvariants}), the integral (\ref{intconj1}) is the same whether the theta series is defined by summation over $X^+(k)$ or over the largest subvariety where $\mathcal Z(X)$ acts faithfully. The reason is a phenomenon that has frequently been observed in the Rankin-Selberg method, namely that the stabilizers of points in all but the open orbit contain unipotent radicals of proper parabolics. Although this is not a feature of the Rankin-Selberg method only, we present the proof there in order not to interrupt the exposition here. 
\end{remark}

\subsection{Geometric models and the basic function}\label{ssgeommodels}

We now discuss the geometric models and explain the first requirement of Conjecture \ref{mainconjecture}.
The models that we are about to discuss are relevant to a spherical variety $X$ over an \emph{equal-characteristic} local field $F$, and are not local, but global in nature.

\subsubsection{The Gaitsgory-Nadler spaces \cite{GN}.}\label{GN} Let $X$ be an affine spherical variety over $\CC$, and let $C$ be a smooth complete complex algebraic curve. Consider the ind-stack $\mathcal Z$ of \emph{meromorphic quasimaps} which, by definition, classifies data:
$$ (c, \mathcal P_G, \sigma)$$
where $c\in C$, $\mathcal P_G$ is a principal $G$-bundle on $C$, and $\sigma$ is a section: $C\smallsetminus\{c\}\to \mathcal P_G\times^G X$ whose image is not contained in $X\smallsetminus X^+$. Clearly, $\mathcal Z$ is fibered over $C$ (projection to the first factor). It is a stack of infinite type, however it is a union of open substacks of finite type, each being the quotient of a scheme by an affine group, and therefore one can define intersection cohomology sheaves on it without a problem.

The same definitions can be given if $G,X$ are defined over a finite field $\FF$.

To any quasimap one can associate an element of $X^+(\mathcal K)/G(\mathfrak O)$ (where $\mathfrak O=\CC[[t]],\mathcal K=\CC((t))$) as follows: Choose a trivialization of $\mathcal P_G$ in a formal neighborhood of $c$ and an identification of this formal neighborhood with $\spec(\mathfrak O)$ -- then the section $\sigma$ defines a point in $X^+(\mathcal K)$, which depends on the choices made. The corresponding coset in $X^+(\mathcal K)/G(\mathfrak O)$ is independent of choices.

This allows us to stratify our space according to the stratification, provided by Theorem \ref{Cstratification}, of $X^+(\mathcal K)/G(\mathfrak O)$. We only describe some of the strata here: For $\theta\in \Lambda_X^+$, let $\mathcal Z^\theta$ denote the quasimaps of the form $(c,\mathcal P_G,\sigma: C\smallsetminus\{c\}\to\mathcal P_G\times^G X^+)$ which correspond to the coset $\theta\in X^+(\mathcal K)/G(\mathfrak O)$ at $c$. Then $\mathcal Z^\theta$ can be thought of as a (global) geometric model for that coset. The \emph{basic stratum} $\mathcal Z^0$ consists of quasimaps of the form $(c\in C, \mathcal P_G, \sigma: C\to \mathcal P_G\times^G X^+)$. Notice that these sub-stacks do not depend on the compactification $X$ of $X^+$. Their \emph{closure}, though, does. For instance, the closure of $\mathcal Z^0$ can be identified with an open substack in the quotient stack $X_C/G_C$ over $C$, namely the stack whose $S$-objects are $S$-objects of $X_C/G_C$ but not of $(X\smallsetminus X^+)_C/G_C$. These are 
the quasimaps for which the corresponding point in $X^+(\mathcal K)/G(\mathfrak O)$ lies in the image of $X^+(\mathcal K)\cap X(\mathfrak O)$. Hence, the closure of $\mathcal Z^0$ should be thought of as a geometric model for $X^+(\mathcal K)\cap X(\mathfrak O)$.

Since the spaces of Gaitsgory and Nadler are global in nature, it is in fact imprecise to say that they are geometric models for local spaces. However, \emph{their singularities} are expected to model the singularities of $G(\mathfrak O)$-invariant subsets of $X^+(\mathcal K)$.

\subsubsection{Drinfeld's compactifications.}\label{Dmodels} The spaces of Gaitsgory and Nadler described above are (slightly modified) generalizations of spaces introduced by Drinfeld in the cases: $X=\overline{U_P\backslash G}^\aff$ or $X=\overline{[P,P]\backslash G}^\aff$, where $P\subset G$ is a proper parabolic and $U_P$ its unipotent radical. The corresponding spaces are denoted by $\widetilde{\Bun_P}$ and $\overline{\Bun_P}$, respectively. Our basic references here are \cite{BG, BFGM}. The only differences between the definition of these stacks and the stacks $\mathcal Z$ of Gaitsgory and Nadler are that the section $\sigma$ has to be defined on all $C$, and it does not have a distinguished point $c$. Therefore, for a quasimap in Drinfeld's spaces and any point $c\in C$ the corresponding element of $X^+(\mathcal K)/G(\mathfrak O)$ has to belong to the cosets which belong to $X(\mathfrak O)$. (These will be described later when we review the computations of \cite{BFGM}.) 

This particular case is very important to us because it is related to Eisenstein series, and moreover the intersection cohomology sheaf of the ``basic stratum'' has been computed (when $G, X$ are defined over $\FF$).

\subsubsection{The basic function}\label{sssbasicfunction}

We return to the setting where $X$ is an affine spherical variety for a split group $G$ over a local, non-archimedean field $F$ whose ring of integers we denote by $\mathfrak o$ and whose (finite) residue field we denote by $\FF$. We assume that $X$, $G$ and the completions $\bar X, \hat X$ introduced before have the properties of Proposition \ref{places} over $\mathfrak o$, and denote $K=G(\mathfrak o)$. The goal is to define the ``basic function'' $\Phi^0$ on $X^+(F)$, which will be $K$-invariant and supported in $X(\mathfrak o)$. This function will factor through the map $X^+(F)/K\to\Lambda_X^+$ of Theorem \ref{stratification}. The idea is to define a function on $\Lambda_X^+$ using equal-characteristic models of $X$.

Define the Gaitsgory-Nadler stack $\mathcal Z$ as in \S\ref{GN} over $\FF$. Since, by assumption, $X_{\FF}$ has a completion $\bar X_{\FF}$ with the properties of Proposition \ref{places} (and, hence, the same holds for the base change $X_{\FF[[t]]}$), the generalized Cartan decomposition \ref{stratification} holds for $G(\FF[[t]])$-orbits on $X^+(F((t)))$: they admit a natural map onto $\Lambda_X^+$. Hence the strata $\mathcal Z^\theta$ of $\mathcal Z$ are well-defined over $\FF$. Let $IC^0$ denote the intersection cohomology sheaf of the closure of the basic stratum $\mathcal Z^0$ (how exactly to normalize it is not important at this point, since we will normalize the corresponding function). We will obtain the value of our function at ${\check\lambda}\in \Lambda_X^+$ as trace of Frobenius acting on the stalk of $IC^0$ at an $\FF$-object $x_{\check\lambda}$ in the stratum $\mathcal Z^{\check\lambda}$. However, since these strata are only locally of finite type, and not of pure dimension, we must be careful 
to make compatible choices of points as ${\check\lambda}$ varies. (It is expected that $IC^0$ is locally constant on the strata -- this will be discussed below.)

The compatibility condition is related to the natural requirement that the action of the unramified Hecke algebra on the functions which will be obtained from sheaves is compatible, via the function-sheaf correspondence, with the action of its geometric counterpart on sheaves. First of all, let us fix a quasimap $x_0=(c_0,\mathcal P_0, \sigma_0)$ in the $\FF$-objects of the basic stratum $\mathcal Z^0$. Now consider the subcategory $\mathcal Z_{x_0}$ of $\mathcal Z$ consisting of $\FF$-quasimaps $(c_0,\mathcal P_G,\sigma)$ with the property that there exists an isomorphism $\iota:\mathcal P_0|_{C\smallsetminus\{c_0\}} \stackrel{\sim}{\rightarrow} \mathcal P_G|_{C\smallsetminus\{c_0\}} $ (inducing isomorphisms between $\mathcal P_0\times^G X$ and $\mathcal P_G\times^G X$, also to be denoted by $\iota$) such that $\sigma=\iota\circ\sigma_0$. Hence, the objects in $\mathcal Z_{x_0}$ are those obtained from $x_0$ via \emph{meromorphic Hecke modifications} at the point $c_0$ \cite[\S 4]{GN}.

For each ${\check\lambda}\in \Lambda_X^+$, pick an object $x_{\check\lambda}\in \mathcal Z_{x_0}$ which belongs to the stratum $\mathcal Z^{\check\lambda}$. We define the \emph{basic function} $\Phi^0$ on $\Lambda_X^+$ to be:
\begin{equation}\label{basicfunction}
\Phi^0({\check\lambda})=  c\cdot \sum_i (-1)^i \tr(\Fr, H^i(IC^0_{x_{\check\lambda}}))
\end{equation}
where $IC^0_{x_{\check\lambda}}$ denotes the stalk of $IC^0$ at $x_{\check\lambda}$ and $\Fr$ denotes the geometric Frobenius. The constant $c$ (independent of ${\check\lambda}$) is chosen such that $\Phi^0(0)=1$. 

Now we return to $X(F)$ and we identify $\Phi^0$ with a $K$-invariant function on $X^+(F)$ (also to be denoted by $\Phi^0$) via the stratification of Theorem \ref{stratification}.

This is the ``basic function'' of Conjecture \ref{mainconjecture} at the given place. The definition implies that the support of the basic function is contained in $X(\mathfrak o)$, since the closure of the basic stratum includes the stratum $\mathcal Z^\theta$ only if $\theta$ corresponds to a $G(\mathfrak o)$-orbit belonging to $X(\mathfrak o)$. The independence of choices of the basic function is widely expected but, in the absence of suitable finite-dimensional geometric models, not known. We impose it as an assumption, together with other properties that should naturally follow from the properties of intersection cohomology if one had suitable local models. Notice that for $X=\overline{U_P\backslash G}^\aff$ or $X=\overline{[P,P]\backslash G}^\aff$, one could have used instead the Drinfeld models of \ref{Dmodels} to define the basic function.

\begin{assumption}\label{mainassumption} \begin{enumerate}
\item
 The basic function $\Phi^0$ on $X^+(F)$ is well-defined and independent of:
\begin{itemize}
 \item the choices of objects $x_{\check\lambda}$;
 \item (if $X=\overline{U_P\backslash G}^\aff$ or $X=\overline{[P,P]\backslash G}^\aff$) which model of \S\ref{ssgeommodels} one uses to define them;
 \item the group $G$ acting on $X$; more precisely, if $G_1, G_2$ act on $X$ and we denote by $X_1^+,X_2^+$ the open orbits, then the restriction of $\Phi^0$ to $X_1^+(F)\cap X_2^+(F)$ should be the same.
\end{itemize}
\item If $Z$ is an affine homogeneous spherical $G$-variety and $p:X\to Z$ a surjective equivariant morphism then the basic function on $X$, evaluated at any point $x\in X^+(F)\cap X(\mathfrak o)$, is equal to the basic function of the fiber of $p$ over $p(x)$ (considered as a $G_{p(x)}$-spherical variety).
                                         \end{enumerate}
\end{assumption}

We also discuss how to deduce the growth assumption on elements of the Schwartz space (\S \ref{ssformalism}) for the basic function. Assume now that $X$ is defined globally over a number field $k$, and fix a finite set of places $S_0$ and suitable $\mathfrak o_{S_0}$-models as in Proposition \ref{places}. Recall (\S \ref{sssdistance}) that these models define a distance function $d_Z=\prod_{v\notin S_0} d_{Z,v}$ from $Z=X\smallsetminus X^+$ on $\prod_{v\notin S_0} X(\mathfrak o_v)$. 

\begin{proposition}
Assume that there are a $\chi\in \varchi(X)\otimes\RR$ such that for all places $v$ and all $\check\lambda\in \Lambda_X^+$:
$$|\Phi_v^0(\check\lambda)|\le q_v^{\left<\chi,\check\lambda\right>}$$ (where $q_v=|\FF_v|$). Then there is a natural number $n$ such that:
$$\left|\prod_{v\notin S_0} \Phi_v^0 (x) \right|\le (d_Z(x))^{-n}$$
for all $x \in X^+(\Ad^{S_0})$.
\end{proposition}

Here $\Ad^{S_0}$ denotes the adeles outside of $S_0$. Of course, the function is zero off $\prod_{v\notin S_0} X(\mathfrak o_v)$ so the extension of the distance function off integral points of $X$ plays no role in the statement.

\begin{proof}
First of all, we claim:
\begin{quote}
 The local distance function $d_{Z,v}$ on $X(\mathfrak o_v)$ is $G(\mathfrak o_v)$-invariant.
\end{quote}
 Indeed, $G(\mathfrak o_v)$ preserves the ideal of $Z$ in $\mathfrak o_v[X]$ and therefore its image in $\mathfrak o_v$ under any $\mathfrak o_v$-point.

Hence, since both $d_Z$ and $\prod_{v\notin S_0} \Phi_v^0$ are $\prod_{v\notin S_0} G(\mathfrak o_v)$-invariant, it suffices to prove the proposition for a set of representatives of $\prod_{v\notin S_0} G(\mathfrak o_v)$-orbits in the support of $\prod_{v\notin S_0} \Phi_v^0$, namely elements of $A_X(\Ad^{S_0})$ which at every place $v$ have image in $\check\Lambda_X^+\cap \mathcal C(X)$.

Let $Y$ denote the ``standard $\mathfrak o_{S_0}$-model'' of the affine toric embedding of $A_X$ corresponding to the cone $\check\Lambda_X^+\cap \mathcal C(X)$. By assumption (see Proposition \ref{places}), there is a morphism $Y\to X$. Therefore, if $Y_1$ denotes the complement of the open orbit on $Y$, the corresponding distance functions on $A_X(k_v)$, for every $v\notin S_0$, compare as: $d_{Z,v}\le d_{Y_1,v}$. On the other hand, clearly, for every $\chi\in \varchi(X)\otimes\RR$ there is a natural number $n$ such that for all $v\notin S_0$ we have $d_{Y_1,v}^{-n}\ge q_v^{\left<\chi,\check\lambda\right>}$ on $A_X(k_v)\cap Y(\mathfrak o_v)$. The claim follows.

\end{proof}


\section{Periods and the Rankin-Selberg method} \label{secrs}

\subsection{Pre-flag bundles} \label{sspreflag}

We are about to describe the geometric structure which gives rise to Rankin-Selberg integrals. We hasten to clarify, and it will probably be clear to the reader, that it is not a very natural structure from the general point of view that we have taken thus far, and its occurence should be seen as a coincidence. Indeed, the structure is not defined in terms of the original group $G$, but in terms of a possibly different group $\tilde G$, and relies on being able to decompose the variety by a sequence of maps with simple, easily identifiable fibers.

We keep assuming that $\mathcal Z(G)^0\xrightarrow{\sim} \mathcal Z(X)$. We will say that an affine spherical $G$-variety $X$ has the structure of a \emph{pre-flag bundle} if it is the affine closure of a $G$-stable subvariety $\tilde X^+$, which has the following structure:
\begin{enumerate}
 \item $\tilde X^+$ is homogeneous under a reductive group $\tilde G$;
 \item there is a diagram of homogeneous $\tilde G$-varieties with surjective morphisms:
\begin{equation*}
 \begin{CD}
  \tilde X^+ \\
  @VV{L\text{-torsor}}V \\
  \tilde Y \\
  @VV{\text{fiber over }y\in Y\text{ is a flag variety for }\tilde G_y}V \\
  Y&\text{ (}\simeq G'_y\backslash G' \simeq \tilde G_y\backslash \tilde G\text{ with }G'_y, \tilde G_y\text{ reductive).}
 \end{CD}
\end{equation*}
where:
\begin{itemize}
 \item $Y$ is an affine, $\tilde G$-homogeneous variety (called the \emph{base} of the pre-flag bundle);
 \item $\tilde Y$ is proper over $Y$ (hence the fiber over $y\in Y$ is a flag variety for $\tilde G_y$);
 \item $\tilde Y$ is the quotient of $\tilde X^+$ by the free, $\tilde G$-equivariant action of a reductive group $L$ which contains $\mathcal Z(X)$;
 \item $L$ is an almost direct factor of $G$.
\end{itemize}
\end{enumerate}

\begin{remark}The group $G'$ has been inserted in the diagram for later reference. It is supposed to belong to an almost direct decomposition $G = L\cdot G'$ and it necessarily acts transitively on $Y$, since $\mathcal Z(X)$ acts trivially on $Y$ while, on the other hand, it retracts all points onto a homogeneous subvariety by Proposition \ref{structureaffine}.
\end{remark}

Hence, the notion of a pre-flag bundle combines the notion of a homogeneous affine variety (which here is the base $Y$), with the notion of a \emph{pre-flag variety}, i.e.\ a quasi-affine quotient of $N''\backslash G''$ by a subgroup of $M''$, where $M''N''$ is the Levi decomposition of a parabolic of $G''$ (here, the fibers over $Y$ are such,\footnote{Notice that $L$ is necessarily a quotient of a Levi subgroup of $\tilde G_y$. Indeed, if we write as $\tilde X^+_y=\tilde H_y\backslash\tilde G_y\to \tilde P_y\backslash\tilde G_y$ the map between the fibers of $\tilde X^+$, resp.\ $\tilde X^+/L$ over $y\in Y$, where $\tilde P_y$ is a parabolic of $\tilde G_y$, then $L$ can be identified with a subgroup of $\Aut^{\tilde G_y}(X_y)$ preserving the fiber of this map, that is with a subgroup of $\mathcal N_{\tilde P_y}(\tilde H_y)/\tilde H_y$. Since it acts transitively on the fibers of this map, it follows that $\tilde H_y$ must be normal in $\tilde P_y$, and $L$ must be the quotient $\tilde P_y/\tilde H_y$. 
Since $L$ is reductive, this also implies that $\tilde H_y$ contains the unipotent radical of $\tilde P_y$. 
} setting $G''$=the stabilizer of a point on $Y$). Of course, each of these constituents can be trivial, for instance $Y$ can be a point (in which case we are dealing with a pre-flag variety, but possibly for a different group than $G$), or $X$ could be equal to $Y$ (in which case we are dealing with affine homogeneous varieties).

In this paper we will additionally impose the condition, without mentioning it further, that the fiber $\tilde X^+_y$ over $y\in Y$ is a product of varieties $[P_i,P_i]\backslash G_i$ or of the form $U_{P_i}\backslash G_i$, where $\prod_i G_i = \tilde G_y$. This extra condition will allow us to restrict our discussion to Eisenstein series induced either from cusp forms or from characters of parabolic subgroups, and to use the computations of \cite{BFGM}. Notice that the dual group of $L$ acts on the unipotent radical of the dual parabolic to $\tilde P_y$ inside of the dual group of $\tilde G_y$; indeed the quotient $\tilde P_y\twoheadrightarrow L$ gives rise to a homomorphism: $\check L\to \check{\tilde L}_y$, where $\check{\tilde L}_y$ is the standard Levi dual to $\tilde P_y$. We let $\mathfrak{\check u}_{\tilde P}$ denote\footnote{It would be more correct to consider only what will later be denoted by $\mathfrak{\check u}_{\tilde P}^f$ for those factors of $\tilde X^+_y$ that are of the form $[P_i,P_i]\
backslash G_i$, but that does not make any difference for the statement of Theorem \ref{preflagthm} below, since we are only using $\mathfrak{\check u}_{\tilde P}$ to require the meromorphic continuation of an $L$-function, and the difference if we took $\mathfrak{\check u}_{\tilde P}^f$ instead would just be some abelian $L$-function.} the Lie algebra of the unipotent radical of the parabolic dual to $\tilde P_y$, considered as a representation of $\check L$.


The requirement that $\tilde G$ commutes with the action of $\mathcal Z(X)$ (by the condition $\mathcal Z(X)\subset L$) is meant to allow us to the $\mathcal Z(X)$-Mellin transforms of $X$-theta series to usual Eisenstein series on $\tilde G_y$ induced from $\tilde P_y$.

\begin{example}\label{vsexample}
 The variety $\Mat_n$ for $\GL_n\times\GL_n$ ($n\ge 2$) is a pre-flag variety, and more generally so is any $N$-dimensional vector space (here $N=n^2$) with a linear $G$-action, as it is equal to the affine closure of $P_{N}\backslash \GL_{N}$ (with $P_N$ the mirabolic subgroup). Notice, however, that an $n+m$-dimensional vector space ($n,m\ge 2$) can be considered as a pre-flag variety for both $\tilde G=\GL_{n+m}$ and $\tilde G=\GL_n\times \GL_m$; which one we will choose will depend on which torus action we will consider (i.e.\ what is $\mathcal Z(X)$). For instance, for the second possibility, decomposing the given vector space as $X=V=V_n\oplus V_m$ we have:
\begin{enumerate}
 \item $Y$= a point;
 \item $\tilde X^+ = (V_n\smallsetminus\{0\}) \times (V_m\smallsetminus\{0\})$;
 \item $\tilde G = \GL(V_n)\times GL(V_m)$;
 \item $L=\mathcal Z(X) = \Gm\times \Gm$, the two copies acting on $V_n$ and $V_m$, respectively;
 \item and we can take $G=\tilde G$ (with $L$ identified as its center), or any subgroup thereof which contains the center and acts spherically.
\end{enumerate}
\end{example}

From our point of view, whether a spherical variety is a pre-flag bundle or not is a matter of ``chance'' and in fact should be irrelevant as far as properties of $X$-theta series and their applications go -- the fundamental object is just $X$ as a $G$-variety, and not its structure of a pre-flag bundle. We will try to provide support for this point of view in \S \ref{sstensor}. However, in absence of a general proof of Conjecture \ref{mainconjecture}, this is the only case where its validity, in the weaker form of Conjecture \ref{weakconjecture}, can be proven. Moreover, the concept of pre-flag bundles is enough to explain a good part of the Rankin-Selberg method.

We assume throughout in this section that the local Schwartz spaces $\mathcal S(X_v)$ are the $G$-spaces generated by the ``basic function'', which we extract from computations on Drinfeld spaces (outside of a finite number of places), and by functions in $C_c^\infty(X_v^+)$ obtained as convolutions of delta functions with smooth, compactly supported measures on $G_v$. (At non-archimedean places, such functions span $C_c^\infty(X_v)$.) The main result of this section is the following:

\begin{theorem}\label{preflagthm}
Assume that $X$ is a wavefront spherical variety with trivial arithmetic multiplicity which has the structure of a a pre-flag bundle, and let $\tau$ vary over a holomorphic family of cuspidal automorphic representations of $G$ (i.e.\ an irreducible cuspidal representation twisted by idele class characters of the group). Let $\tau_1$ denote the isomorphism class of the restriction of $\tau$ to $L$, and assume that for some finite set of places $S$ the partial $L$-function $L^{S}(\tau_1,{\mathfrak{\check u}_{\tilde P}}, 1)$  has meromorphic continuation everywhere (as $\tau$ varies in this family).

Then Conjecture \ref{weakconjecture} holds for $\phi\in \tau$ and $\mathcal S(X_v)$ as described above.
\end{theorem}

We prove this theorem in \S \ref{ssRS} by appealing to the meromorphic continuation of usual Eisenstein series, after explicitly describing the basic vectors according to the computations of intersection cohomology sheaves on Drinfeld spaces in \cite{BFGM}. However, the application of the meromorphic continuation of Eisenstein series is not completely trivial as in some cases we have to use the theory of spherical varieties to show that as we ``unfold'' this integral certain summands vanish (in the language often used in the theory of Rankin-Selberg integrals: certain $G$-orbits on $X$ are ``negligible''). We start by demonstrating an extreme case, which gives rise to period integrals.

\subsection{Period integrals}\label{ssperiods}

First consider the extreme case of a pre-flag bundle with trivial fibers: Namely, choosing a point $x_0\in X(k)$, we have $X=H\backslash G$ with $H=G_{x_0}$ \emph{reductive}. Then at each place $v\notin S_0$ the basic function is the characteristic function of $X(\mathfrak o_v)$, and we may assume that $\mathcal S(X(\Ad)) = C_c^\infty(X(\Ad))$. The multiplicity-freeness assumption of \S \ref{ssconjectures} implies, in particular, that at almost every place $G(\mathfrak o_v)$ acts transitively on $X(\mathfrak o_v)$. Then we can take $\Phi \in \mathcal S(X(\Ad))$ of the form $\Phi=h\star \delta_{x_0}$ where $h\in \mathcal H(G(\Ad))$, the Hecke algebra of compactly supported smooth measures on $G(\Ad)$, and $\delta_{x_0}$ is the delta function at $x_0$ (considered as a generalized function).

Then, if $\check h$ denotes the adjoint to $h$ element of $\mathcal H(G(\Ad))$, the integral:
$$\int_{G(k)\backslash G(\Ad)} \phi\cdot\omega(g) \theta(\Phi,g) dg$$
 of Conjecture \ref{weakconjecture} is equal to:
\begin{equation}\int_{H(k)\backslash H(\Ad)} (\check h\star \phi)\cdot  \omega(g) dg.\end{equation}
This is called a \emph{period integral}, and such integrals have been studied extensively. Hence \emph{period integrals are the special case of the pairing of Conjecture \ref{weakconjecture} which is obtained from pre-flag bundles with trivial fibers} (i.e.\ affine homogeneous spherical varieties).

For example, when $X=\GL_2$, $G=\Gm\times\GL_2$, with $\Gm$ acting as a non-central torus of $\GL_2$ by multiplication on the left, we get the period integral of Hecke (\ref{hecke}), discussed in the introduction.
All spherical period integrals are included in the lists of Knop and van Steirteghem \cite{KnVS} which we will discuss in the next section.

\subsection{Connection to usual Eisenstein series}

\subsubsection{Certain stacks and sheaves related to flag varieties} The goal of this subsection is to explicate the basic functions $\Phi_v^0$ for pre-flag bundles, based on the computations of \cite{BFGM}. We work with the varieties $X=\overline{[P,P]\backslash G}^\aff$ or $X=\overline{U_P\backslash G}^\aff$ and use the notation of \S \ref{Dmodels}. We do not aim to give complete definitions of the constructions of \emph{loc.\ cit.}, but to provide a guide for the reader who would like to extract from it the parts most relevant to our present discussion. The final result will be the following formula for the basic function $\Phi^0$ (locally at a non-archimedean place, which we suppress from the notation):

\begin{theorem} \label{BFGMfunctions}
For $X=\overline{H\backslash G}$ in each of the following cases, we have:
\begin{itemize}
 \item If $H=U_P$:  $\Phi^0 = \sum_{i\ge 0} q^{-i}\widecheck\Sat_M\left(\operatorname{Sym}^i(\check{\mathfrak u}_P)\right)\star 1_{HK} = $ \begin{equation}=\widecheck\Sat_M\left(\frac{1}{\wedge^{\operatorname{top}} (1-q^{-1}\check{\mathfrak u}_P)}\right) \star 1_{HK}.\end{equation}
 \item If $H=[P,P]$: $\Phi^0 = \sum_{i\ge 0} q^{-i}\widecheck\Sat_{M^\ab}\left(\operatorname{Sym}^i(\check{\mathfrak u}_P^f)\right)\star 1_{HK} =$ \begin{equation}= \widecheck\Sat_{M^\ab}\left(\frac{1}{\wedge^{\operatorname{top}}(1-q^{-1}\check{\mathfrak u}_P^f)}\right) \star 1_{HK}.\end{equation}
\end{itemize}
\end{theorem}
Here $\widecheck\Sat$ denotes the power series in the Hecke algebra associated by the Satake isomorphism to the given power series in the representation ring of the dual group -- it will be explained in detail in \S \ref{sssEisfunctions}. 

We denote by $\Lambda_{G,P}$ the lattice of cocharacters of the torus $M/[M,M]$ and by $\Lambda_{G,P}^\pos$ the sub-semigroup spanned by the images of $\check\Delta\smallsetminus\check\Delta_M$. 
For every $\theta\in \Lambda_{G,P}^\pos$ we have a canonical locally closed embedding: $j_\theta: C\times \Bun_P\to \overline{\Bun}_P$ \cite[Proposition 1.5]{BFGM}. The image will be denoted by ${_{(\theta)}\overline{Bun_P}}$. (Notice: This is not the same as what is denoted in \emph{loc.cit.} by ${_{\theta}\overline{Bun_P}}$, but rather what is denoted by ${_{\mathfrak U(\theta)}\overline{Bun_P}}$, when $\mathfrak U(\theta)$ is the trivial partition of $\theta$.) Its preimage in $\widetilde{\Bun}_P$ will be denoted by $_{(\theta)}\widetilde{\Bun}_P$. We have a canonical isomorphism: $_{(\theta)}\widetilde{\Bun}_P\simeq \Bun_P\times_{\Bun_M}\mathcal H_M^{(\theta)}$, where $\mathcal H_M^{(\theta)}$ is a stack which will be described below.

\begin{remarks}\begin{enumerate}[(i)]\item
If $X=\overline{[P,P]\backslash G}^\aff$ under the $M^\ab=M/[M,M]\times G$-action, then $\Lambda_X^+$ can be identified with $\Lambda_{G,P}$, and $_{(\theta)}\overline{Bun}_P$ is precisely the analog of what we denoted by $\mathcal Z^{w_0\theta}$ on the Gaitsgory-Nadler stacks, where $w_0$ is the longest element in the Weyl group of $G$. The reason that only $\theta\in \Lambda_{G,P}^\pos$ appear is that, as we remarked in \S\ref{Dmodels}, the quasi-maps on Drinfeld spaces are, by definition, not allowed to have poles. For the reader who would like to trace this back to the combinatorics of quasi-affine varieties and their affine closures of \S\ref{ssaffine} we mention that the cone spanned by $\rho(\mathcal D)$ is the cone spanned by the images of $\check\Delta\smallsetminus\check\Delta_M$. 

\item If $X=\overline{U_P\backslash G}^\aff$ under the $M\times G$-action then $\Lambda_X^+\simeq\{{\check\lambda}\in \Lambda_A| \left<{\check\lambda},\alpha\right>\le 0 \text{ for all }\alpha\in\Delta_M\}$ (where we denote by $A$ the maximal torus of $G$ and by $\Lambda_A$ its cocharacter lattice). There is a map: $\Lambda_X\to\Lambda_{G,P}$, and $_{(\theta)}\widetilde{Bun}_P$ corresponds to the union of the strata $\mathcal Z^{w_0{\check\lambda}}$ of Gaitsgory-Nadler, with ${\check\lambda}$ ranging over all the $M$-dominant preimages of $\theta$.\end{enumerate}                                                                                                                                                                                                                                                                                                                                                                                                                                                                       
                                              \end{remarks}

We have the geometric Satake isomorphism, i.e.\ a functor $\operatorname{Loc}:\Rep(\check G)\to \operatorname{Perv}(\mathcal G_G)$ such that the irreducible representation of $\check G$ with highest weight ${\check\lambda}$ goes to the intersection cohomology sheaf of a $G(\mathfrak o)$-equivariant closed, finite-dimensional subscheme $\overline{\mathcal G_G}^{\check\lambda}$. We will make use of this functor for $M$, rather than $G$. If $V$ is a representation of $\check M$ -- assumed ``positive''; this has to do with the fact that we don't allow poles, but there's no need to explain it here -- and $\theta\in \Lambda_{G,P}^\pos$ then we define $\operatorname{Loc}^{(\theta)}(V)$ to be $\operatorname{Loc}(V_\theta)$, where $V_\theta$ is the $\theta$-isotypic component of $V$. (We ignore a twist by $\overline{\QQ_l}[1]\left(\frac{1}{2}\right)^{-1}$ introduced in \cite{BFGM}, and modify the results accordingly.) 

We now introduce relative, global versions of the above spaces. We denote by $\mathcal H_M$ the \emph{Hecke stack} of $M$. It is related to $\mathcal G_M$ as follows: If we fix a curve $C$ and a point $x\in \mathcal C$ then, by definition, $\mathcal G_M$ is the functor Schemes$\to$Sets which associates to every scheme $S$ the set of pairs $(\mathcal F_G,\beta)$ where $\mathcal F_M$ is a principal $M$-bundle over $C\times S$ and $\beta$ is an isomorphism of it outside of $(C\smallsetminus\{x\})\times S$ with the trivial $M$-bundle. The relative version of this, as we allow the point $x$ to move over the curve, is denoted by $\mathcal G_{M,C}$, and the relative version of the latter, as we replace the trivial $M$-bundle with an arbitrary $M$-bundle, is $\mathcal H_M$. It is fibered over $C\times\Bun_M$.

In \emph{loc.cit.}, p.\ 389, certain closed, finite-dimensional subschemes $\mathcal G_M^{+,\theta}$ of $\mathcal G_M$ are defined for every $\theta\in \Lambda_{G,P}^\pos$ which at the level of reduced schemes are isomorphic to $\overline{\mathcal G_M}^{\flat(\theta)}$, where $\flat(\theta)$ is an $M$-dominant coweight associated to $\theta$ -- the ``least dominant'' coweight mapping to $\theta$. The relative versions of those give rise to substacks $\mathcal H_M^{(\theta)}$ of $\mathcal H_M$. 

For these relative versions we have: Functors $\operatorname{Loc}_{\Bun_M,C}$ (resp.\ $\operatorname{Loc}_{\Bun_M,C}^{(\theta)}$) from $\Rep(\check M)$ to perverse sheaves on $\mathcal H_M$ (resp.\ $\mathcal H_M^{(\theta)}$) and $\operatorname{Loc}_{\Bun_P,C}$ (resp.\ $\operatorname{Loc}_{\Bun_P,C}^{(\theta)}$) to perverse sheaves on $\Bun_P\times_{\Bun_M}\mathcal H_M$ (resp.\ $\Bun_P\times_{\Bun_M}\mathcal H_M^{(\theta)}$), the latter being $IC_{\Bun P}$ along the base $\Bun_P$. 

Then the main theorem of \cite{BFGM} (Theorem 1.12) is a description of the $*$-restriction of $IC_{\widetilde{\Bun}_P}$ to $_{(\theta)}\widetilde{\Bun}_P\simeq \Bun_P\times_{\Bun_M}\mathcal H_M^{(\theta)}$. Moreover, Theorem 7.3 does the same thing for $IC_{\overline{\Bun}_P}$ and $_{(\theta)}\overline{\Bun}_P\simeq C\times \Bun_P$. The normalization of $IC$ sheaves is that they are pure of weight 0; i.e.\ for a smooth variety $Y$ of dimension $n$ we have $IC_Y\simeq \left(\overline{\QQ_l}\left(\frac{1}{2}\right)[1]\right)^{\otimes n}$, where $\overline{\QQ_l}\left(\frac{1}{2}\right)$ is a fixed square root of $q$.

\begin{theorem}[\cite{BFGM}, Theorems 1.12 and 7.3]\label{BFGMtheorem}
 The $*$-restriction of $IC_{\widetilde{\Bun}_P}$ to $_{(\theta)}\widetilde{\Bun}_P\simeq \Bun_P\times_{\Bun_M}\mathcal H_M^{(\theta)}$ is equal to:
\begin{equation}\label{tildebun}
 \operatorname{Loc}_{\Bun_P,C}^{(\theta)} \left(\oplus_{i\ge 0} \operatorname{Sym}^i(\check{\mathfrak u}_P)\otimes \overline{\QQ_l}(i)[2i]\right).
\end{equation}
The *-restriction of $IC_{\overline{\Bun}_P}$ to $_{(\theta)}\overline{\Bun}_P\simeq C\times \Bun_P$ is equal to:
\begin{equation}\label{linebun}
 IC_{_{(\theta)}\overline{\Bun}_P} \otimes \overline{\operatorname{Loc}}\left(\oplus_{i\ge 0} \operatorname{Sym}^i(\check{\mathfrak u}_P^f)_\theta\otimes \overline{\QQ_l}(i)[2i]\right).
\end{equation}
\end{theorem}
Here $\check{\mathfrak u}_P$ denotes the adjoint representation of $\check M$ on the unipotent radical of the parabolic dual to $P$. Moreover, $\check{\mathfrak u}_P^f$ denotes the subspace which is fixed under the nilpotent endomorphism $f$ of a principal $\mathfrak{sl}_2$-triple $(h,e,f)$ in the Lie algebra of $\check M$. For the definition of $\overline{\operatorname{Loc}}(V)$, which takes into account the grading on $V$ arising from the $h$-action, cf.\ \emph{loc.cit.}, \S7.1. 

\subsubsection{The corresponding functions}\label{sssEisfunctions}

Let us fix certain normalized Satake isomorphisms. As before, our local, non-archimedean field is denoted by $F$, its ring of integers by $\mathfrak o_F$, and our groups are assumed to have reductive models over $\mathfrak o_F$. As usual, we normalize the action of $M(F)$ (resp.\ $M^\ab(F)$) on functions on $(H\backslash G)(F)$ where $H=U_P$ (resp.\ $[P,P]$) so that it is unitary on $L^2((H\backslash G)(F))$: 
\begin{equation}\label{Maction}
 m\cdot f(H(F)g) = \delta_P^{\frac{1}{2}}(m) f(H(F)m^{-1}g),
\end{equation} 
where $\delta_P$ is the modular character of $P$. We let $M_0=M(\mathfrak o_F)$, and normalize the (classical) Satake isomorphism as follows:
\begin{itemize}
 \item For the Hecke algebra $\mathcal H(M,M_0)$ in the usual way: $$\Sat_M:\CC[\check M]^{\check M}\simeq \CC[\Rep\check M]\xrightarrow{\sim} \mathcal H(M,M_0)$$
where $ \CC[\Rep\check M]$ is the Grothendieck algebra over $\CC$ of the category of algebraic representations of $\check M$.
 \item For the Hecke algebra $\mathcal H(M^\ab,M^\ab_0)$ we shift the usual Satake isomorphism: $\mathcal H(M^\ab,M^\ab_0)\simeq \CC[\mathcal Z(\check M)]\simeq \CC[\Rep \mathcal Z(\check M)]$ by $e^{-\rho_M}$, where $\rho_M$ denotes half the sum of positive roots of $M$. In other words, if $h$ is a compactly supported measure on $M(F)/M_0$, considered (canonically) as a linear combination of cocharacters of $M^\ab$ and hence as a regular function $f$ on the center $\mathcal Z(\check M)$ of its dual group, then we will assign to $h$ the function $z\mapsto f(e^{\rho_M}z)$ on the subvariety $e^{-\rho_M}\mathcal Z(\check M)$ of $\check G$:
$$\Sat_{M^\ab}:\CC[e^{-\rho_M}\mathcal Z(\check M)]\xrightarrow{\sim} \mathcal H(M^\ab,M^\ab_0).$$
\end{itemize}

Let $1_{HK}$ denote the characteristic function of $H\backslash HK$ (where $K=G(\mathfrak o_F)$), and consider the action map: $\mathcal H(M,M_0)\to C_c^\infty((U_P\backslash G)(F))^{M_0\times K}$, respectively $\mathcal H(M^\ab,M^\ab_0)\to C_c^\infty(([P,P]\backslash G)(F))^K$ given by $h\mapsto h \star 1_{HK}$. The map is bijective, and identifies the module $C_c^\infty((H\backslash G)(F))^{M_0\times K}$ with $\CC[\check M]^{\check M}$, resp.\ $\CC[e^{-\rho_M}Z(\check M)]$. Our normalization of the Satake isomorphism is such that this is compatible with the Satake isomorphism for $G$,  $\Sat_G:\mathcal H(G,K)=\CC[\check G]^{\check G}=\CC[\Rep(\check G)]$, in the sense that for $f\in \CC[\check G]^{\check G}$ we have: $$\Sat_G(f)\star 1_{HK} = \widecheck\Sat_{M\text{ or }M^\ab}(f) \star 1_{HK}.$$

Here and later, by the symbol $\check h$ we will be denoting the adjoint of the element $h$ in a Hecke algebra. Its appearance is due to the the definition (\ref{Maction}) of the action of $M$ as a right action on the space and a left action on functions. We extend the ``Sat'' notation to the fraction field of $\CC[\Rep\check M]$ (and, respectively, of $\CC[e^{-\rho_M}\mathcal Z(\check M)]$), where $\Sat_{M \text{ or } M^\ab}(R)$ (with $R$ in the fraction field) is thought of as a power series in the Hecke algebra.

Returning to the Drinfeld spaces discussed in the previous subsection, let $\operatorname{Ff}(E)(x):=\sum_i (-1)^i \tr(\Fr, H^i(E_x))$ denote the alternating sum of the trace of Frobenius acting on the homology of the stalks of a perverse sheaf ($\operatorname{Ff}$ stands for ``faisceaux-fonctions''). As in \S \ref{sssbasicfunction}, we fix an object $x_0$ on the basic stratum, a point $c_0\in C$ (recall that in the definition of Drinfeld's spaces, quasimaps do not have distinguished points) and we evaluate $\operatorname{Ff}(E)$, where $E=IC_{\widetilde{\Bun}_P}$ or $IC_{\overline{\Bun}_P}$, only at objects  $x_{\check\lambda}$ which are obtained by $M\times G$-Hecke modifications at $c_0$. This way, and using the Iwasawa decomposition, we obtain our basic function $\Phi^0$, which is an $M_0\times K$-invariant function on $(H\backslash G)(F)$. Recall that it is by definition normalized such that $\Phi^0(H\backslash H1)=1$. 

The study of the Hecke corresponences in \cite{BG} implies that $$\operatorname{Ff}(\operatorname{Loc}_{\Bun_P, C} (V)) = {\widecheck\Sat}_M(V) \star \operatorname{Ff}(\operatorname{Loc}_{\Bun_P, C} (1))$$ if $H=U_P$, and $$\operatorname{Ff} (\overline{\operatorname{Loc}} (V)) = \widecheck\Sat_{M^\ab}(V) \star \operatorname{Ff}(\overline{\operatorname{Loc}}(1))$$ if $H=[P,P]$. 

\begin{remark}
 The ``unitary'' normalization of the action of $M$ is already present in the sheaf-theoretic setting as follows: Suppose that an object $x_{\check\lambda}$ belongs to $_{({\check\lambda})}\overline{\Bun}_P$ and can be obtained from $x_0$ via Hecke modifications at the distinguished object of $x_0$. Then the dimension of $_{({\check\lambda})}\overline{\Bun}_P\simeq C\times \Bun_P$ at $x_{\check\lambda}$ is $\left<{\check\lambda},2\rho_P\right>$ less than that of $_{(0)}\overline{\Bun}_P$ around $x_0$, where $\rho_P$ denotes the half-sum of roots in the unipotent radical of $P$, i.e.\ $\delta_P=e^{2\rho_P}$. Hence, by the aforementioned normalization of $IC$ sheaves, the contribution of the factor $IC_{({\check\lambda})}\overline{\Bun}_P$  (via Theorem \ref{BFGMtheorem}) to $\Phi^0({\check\lambda})$ will be $q^{\left<{\check\lambda},\rho_P\right>}$ times the contribution of the factor $IC_{(0)}\overline{\Bun}_P$ to $\Phi^0(0)$. Similarly for the strata of $\widetilde{\Bun}_P$.
\end{remark}

Thus, Theorem \ref{BFGMtheorem} translates to the statement of Theorem \ref{BFGMfunctions}:
\begin{itemize}
 \item If $H=U_P$:  $\Phi^0 = \sum_{i\ge 0} q^{-i}\widecheck\Sat_M\left(\operatorname{Sym}^i(\check{\mathfrak u}_P)\right)\star 1_{HK} = $ $$=\widecheck\Sat_M\left(\frac{1}{\wedge^{\operatorname{top}} (1-q^{-1}\check{\mathfrak u}_P)}\right) \star 1_{HK}.$$
 \item If $H=[P,P]$: $\Phi^0 = \sum_{i\ge 0} q^{-i}\widecheck\Sat_{M^\ab}\left(\operatorname{Sym}^i(\check{\mathfrak u}_P^f)\right)\star 1_{HK} =$ $$= \widecheck\Sat_{M^\ab}\left(\frac{1}{\wedge^{\operatorname{top}}(1-q^{-1}\check{\mathfrak u}_P^f)}\right) \star 1_{HK}.$$
\end{itemize}
Notice that in the last expression $\check{\mathfrak u}_P^f$ is considered as a representation of the maximal torus $\check A$ of $\check M$ determined by the principal $\mathfrak{sl}_2$-triple $(h,e,f)$ and, by restricting its character to the subvariety $e^{-\rho_M}\mathcal Z(\check M)$, as an element of $\mathcal H(M^\ab,M^\ab_0)$. This is the case studied in \cite{BK2}, and $\Phi^0$ is the function denoted by $c_{P,0}$ there.

\subsubsection{Connection to Eisenstein series} \label{connES}

Now we discuss our main conjecture when the variety is $X=\overline{U_P\backslash G}^\aff$ or $X=\overline{[P,P]\backslash G}^\aff$ under the (normalized) action of $M\times G$, resp.\ $M^{\textrm{ab}}\times G$. In the latter case, our Eisenstein series $E(\Phi,\omega,g)$ are the usual (degenerate, if $P$ is not the Borel) principal Eisenstein series normalized as in \cite{BK, BK2}, and hence $E(\Phi,\omega,g)$ is indeed meromorphic for all $\omega$.

It will be useful to recall how these meromorphic sections are related to the more usual sections $E(f,\omega,g)$, which are defined in the same way but with $f\in C_c^\infty(([P,P]\backslash G)(\Ad))$. We assume that $\Phi=\prod_v \Phi_v, f=\prod_v f_v$ and $S$ is a finite set of places (including $S_0$) such that for $v\notin S$ we have $\Phi_v=\Phi_v^0$ and $f_v=f_v^0:= 1_{U\backslash G (\mathfrak o_v)}$. Let us also assume for simplicity that for $v\in S$ we have $\Phi_v=f_v$ (a finite number of places certainly do not affect meromorphicity properties). Clearly, for $E(\Phi,\omega,g)$ and $E(f,\omega,g)$ to be non-zero the character $\omega$ must be unramified outside of $S$. Then by the results of the previous paragraph we have:
\begin{equation}\label{degES}
 E(\Phi,\omega,g)= L^{S}(e^{-\rho_M}\omega,\check{\mathfrak u}_P^f, 1) E(f,\omega,g)
\end{equation}
where $L^{S}(e^{-\rho_M}\omega, \check{\mathfrak u}_P^f,1)$ denotes the value at $1$ of the partial (abelian) $L$-function corresponding to the representation $\check{\mathfrak u}_P^f$, whose local factors (at each place $v$) are considered as rational functions on the maximal torus $\check A\subset \check M$ and evaluated  at the point $e^{-\rho_M}\omega_v\in e^{-\rho_M}\mathcal Z(\check M)\subset \check A$.

Now let us consider the case $X=\overline{U_P\backslash G}^\aff$. We let $\tau$ vary over a holomorphic family of cuspidal representations of $M\times G$ and let $\tau\mapsto \phi_\tau$ be a meromorphic section; write $\tau=\tau_1\otimes\tau_2$ according to the decomposition of the group $M\times G$, and assume that, accordingly, $\phi_\tau=\phi_{\tau_1}\otimes\phi_{\tau_2}$, a pure tensor. Assume for the moment that the central character of $\tau$ is sufficiently $X$-positive. If in the notation of Conjecture \ref{weakconjecture} we replace the group $G$ by the group $M\times G$, and perform the integration of the conjecture, but only over the factor $M(k)\backslash M(\Ad)$, then this integral can be written as:
\begin{eqnarray}\label{nprI}
\int_{M(k)\backslash M(\Ad)} \phi_\tau (m,g) \theta (\Phi,(m,g)) dm= \nonumber \\ = \phi_{\tau_2}(g) \int_{M(k)\backslash M(\Ad)} \phi_{\tau_1} (m) \theta (\Phi,(m,g)) dm.
\end{eqnarray}
It is valued in the space of functions on $G(k)\backslash G(\Ad)$. If $\Eis: I_{P(\Ad)}^{G(\Ad)}(\tau_1) \to C^\infty(G(k)\backslash G(\Ad))$ denotes the usual Eisenstein operator, then by unfolding the last integral we see that it is equal to the Eisenstein series:
\begin{equation}\label{nprES}
E_M(\Phi,\phi_1,g):= \Eis\left( \int_{M(\Ad)} \phi_{\tau_1}(m) (m\cdot \Phi) dm\right)(g)
\end{equation}
hence the connection to usual Eisenstein series.

\begin{proposition}\label{prES}
Assume that the partial $L$-function $L^{S}(\tau_1,{\mathfrak{\check u}_P}, 1)$ (for some large enough $S$) has meromorphic everywhere as $\tau_1$ is twisted by characters of $M$. Then the expression (\ref{nprI}) admits meromorphic continuation to all $\tau_1$.
\end{proposition}
\begin{proof}
By the meromorphic continuation of Eisenstein series, it is enough to show that the integral $(\Phi,\phi_{\tau_1})\mapsto \int_{M(\Ad)} \phi_{\tau_1}(m)(m\cdot \Phi)  dm$, which represents a morphism: $\iota_{\tau_1}:\mathcal S(U_P\backslash G(\Ad))\to I_{P(\Ad)}^{G(\Ad)}(\tau_1)$, admits meromorphic continuation in $\tau_1$. This would be the case if $\Phi$ was in $C_c^\infty(U_P\backslash G(\Ad))$. The analogous morphism:  $C_c^\infty(U_P\backslash G(\Ad))\to I_{P(\Ad)}^{G(\Ad)}(\tau_1)$ will also be denoted by $\iota_{\tau_1}$.

Again, we let $S$ be a finite set of places containing $S_0$ and take functions $\Phi=\prod \Phi_v\in \mathcal S(U_P\backslash G(\Ad))$ and $f=\prod_v f_v\in C_c^\infty(U_P\backslash G(\Ad))$ such that for $v\notin S$  $\Phi_v=\Phi_v^0$ is the basic $M_0\times K$-invariant function of the previous paragraph, $f_v=f_v^0=1_{U_PK}$ and for $v\in S$ we have $\Phi_v=f_v$ (for simplicity). Moreover, we assume that $\tau_1$ is unramified for $v\notin S$, otherwise the integral will be zero.

We saw previously that $\Phi_v^0 = \widecheck\Sat_M\left(\frac{1}{\wedge^{\operatorname{top}}(1-q^{-1}\check{\mathfrak u}_P)}\right) \star f_v^0$. By definition of the Satake isomorphism and the equivariance of $\iota_\tau$, in the domain of convergence we have $\iota_{\tau_1} (\Phi) =  L^{S}(\tau_1,{\mathfrak{\check u}_P}, 1) \iota_{\tau_1}(f)$.

Therefore $\Eis(\iota_{\tau_1}(\Phi))= L^{S}(\tau_1,{\mathfrak{\check u}_P}, 1)  \Eis(\iota_{\tau_1}(f))$, and the claim follows from the meromorphic continuation of $\Eis(\iota_{\tau_1}(f))$.
\end{proof}

\begin{remarks}\begin{enumerate}
 \item The meromorphic continuation of $L^{S}(\tau_1,{\mathfrak{\check u}_P}, 1)$ is known in many cases, e.g.\ for $G$ a classical group and $\tau$ generic, by the work of Langlands, Shahidi and Kim, cf.\ \cite{CKM}.
 \item Notice that, as was also observed in \cite{BK,BK2}, the Eisenstein series (\ref{nprES}) has normalized functional equations without $L$-factors.
\end{enumerate}
\end{remarks}

\subsection{The Rankin-Selberg method} \label{ssRS}

According to \cite[\S 5]{BuRS}, the Rankin-Selberg method involves a cusp form on $G$ and an Eisenstein series on a group $\tilde G$, where we have either an embedding: $G\hookrightarrow\tilde G$ or an embedding $\tilde G\hookrightarrow G$, or ``something more complicated''. We certainly do not claim to explain all constructions which have been called ``Rankin-Selberg integrals'', but let us see how a large part\footnote{The multiplicity-one property that seems to underlie almost every integral representation of an $L$-function can be achieved by non-spherical subgroups if we put extra restrictions on the representations we are considering. For example, in the construction of the symmetric square $L$-function by Bump and Ginzburg \cite{BuGi}, we have $H=$ the diagonal copy of $\GL_n$ in $\GL_n\times \text{(a central quotient of)} \widetilde{\GL_n}^2$, where $\widetilde{\GL_n}$ denotes a metaplectic cover, but one restricts to certain ``exceptional'' (and induced-from-exceptional) representations on $\
widetilde{\GL_n}^2$. The examples that our method covers should be seen as the part of the method where such restrictions do not enter.} of this method is covered by our constructions.

Let $X$ be a pre-flag bundle; we will use the notation of \S \ref{sspreflag}. For notational simplicity (the arguments do not change), let us also assume that $L$ is a direct factor of $G$, i.e.\ $G=L\times G'$. Let $\Phi\in \mathcal S(X(\Ad))$. Recall that the $X$-theta series $\theta(\Phi,g)$ has been defined via a sum over $X^+(k)$, where $X^+$ denotes the open $G$-orbit on $X$. On the other hand, to relate our integrals to usual Eisenstein series, we need to sum over $\tilde X^+(k)$, where $\tilde X^+$ is the open $\tilde G$-orbit. Hence, we define:
$$\tilde \theta(\Phi,g)=\sum_{\gamma\in \tilde X^+(k)} \Phi(\gamma\cdot g).$$

We compare the integral of Conjecture \ref{weakconjecture} with the corresponding integral when $\theta$ is substituted by $\tilde \theta$:
\begin{proposition}\label{equal}
Suppose that $X$ is a wavefront spherical variety with the structure of a pre-flag bundle. If $\phi$ is a cusp form on $G$ (with sufficiently $X$-positive central character, so that the following integrals converge) then: \begin{equation}\int_{G(k)\backslash G(\Ad)} \phi(g) \theta(\Phi,g) dg = \int_{G(k)\backslash G(\Ad)} \phi(g) \tilde \theta(\Phi,g) dg.\end{equation}
\end{proposition}

Assume this proposition for now, and let us prove Theorem \ref{preflagthm}; at the same time, we will see that the integral of Conjecture \ref{weakconjecture} is equal to a Rankin-Selberg integral.

Without loss of generality, $\Phi=\prod_v \Phi_v$, and $\phi=\phi_1(l)\phi_2(g)$ according to the decomposition $G=L\times G'$. By Assumption \ref{mainassumption}, and repeating the argument of \S \ref{ssperiods}, we may write $\Phi$ as the convolution with an element $h\in \mathcal H(G'(\Ad))$ of a Schwartz function $\Phi^y$ on $X_y(\Ad)$, where $y\in Y(k)$ and the Schwartz function on $X_y(\Ad)$ is considered as a generalized function on $\tilde X^+(\Ad)$. Then, as in \S \ref{ssperiods}:
$$\int_{G(k)\backslash G(\Ad)} \phi(g) \tilde \theta(\Phi,g) dg = \int_{G_y(k)\backslash G_y(\Ad)} \check h \star \phi (h) \tilde \theta_{\tilde X_y^+}(\Phi^y,h),$$
where $\tilde \theta_{\tilde X_y^+}(\Phi,g)$ denotes the theta series for the $\tilde G_y$-spherical variety $X_y$.

By the decomposition $G=L\times G'$ this is equal to:
$$\int_{G_y'(k)\backslash G_y'(\Ad)} \check h\star \phi_2 (g) \int_{L(k)\backslash L(\Ad)} \phi_1 (l) \tilde \theta_{\tilde X_y^+}(\Phi^y,lg) dl dg.$$

The inner integral is equal to the Eisenstein series $E_L(\Phi, \phi_1, g')$ on the group $\tilde G_y'$, in the notation of (\ref{nprES}), or a degenerate Eisenstein series as in (\ref{degES}), or a product of such\footnote{Rankin-Selberg constructions with products of Eisenstein series have often been encountered in the literature, e.g.\ \cite{BFG, GH}.}, and it has meromorphic continuation under the assumption that $L^{S}(\tau_1,{\mathfrak{\check u}_{\tilde P}}, 1)$ does. Hence, we see that \emph{the integral of conjecture \ref{weakconjecture} is equal to the Rankin-Selberg integral:}
\begin{equation}\label{RSintegral}
\int_{G_y'(k)\backslash G_y'(\Ad)} \check h\star \phi_2(g) E_L(\Phi, \phi_1, g) dg
\end{equation}
and this also completes the proof of Theorem \ref{preflagthm}. In the language of \cite[\S 5]{BuRS}, our formalism combines the appearance of a subgroup $G_y\subset G$ with an embedding of it into another group: $G_y\hookrightarrow \tilde G_y$.

\subsubsection{Proof of Proposition \ref{equal}: Negligible orbits.}

Proposition \ref{equal} will follow from the following statement on the structure of certain spherical varieties:

\begin{proposition}\label{parind}
 If $X$ is a wavefront spherical variety for $G$ with $\Aut^G(X)$ finite, then the isotropy groups of all non-open $G$-orbits contain the unipotent radical of a proper parabolic of $G$.
\end{proposition}

From this, Proposition \ref{equal} follows easily; in the domain of convergence we have:
$$\int_{G(k)\backslash G(\Ad)} \phi(g)\tilde\theta(\Phi,g)= \sum_{\xi\in [\tilde X^+(k)/G(k)]} \int_{G_\xi(k)\backslash G(\Ad)} \phi(g) g\cdot \Phi(\xi) dg$$ where $[\tilde X^+(k)/G(k)]$ denotes any set of representatives for the set of $G(k)$-orbits on $\tilde X^+(k)$. Notice that, by the multiplicity-freeness assumption on $X$, the $k$-points of the open $G$-orbit form a unique $G(k)$-orbit. The summand corresponding to $\xi$ can be written:
$$\int_{G_\xi(\Ad)\backslash G(\Ad)} g\cdot \Phi(\xi) \int_{G_\xi(k)\backslash G_\xi(\Ad)} \phi(hg) dh dg$$
Since $\Aut^G(\tilde X^+/\mathcal Z(X))$ is finite, for $\xi$ in the non-open orbit the stabilizer $G_\xi$ contains the unipotent radical of a proper parabolic by Proposition \ref{parind}, and since $\phi$ is cuspidal the inner integral will vanish. Therefore, only the summand corresponding to the open orbit survives, which folds back to the integral: $$\int_{G(k)\backslash G(\Ad)} \phi(g)\theta(\Phi,g).$$ 

Proposition \ref{parind}, in turn, rests on the following result of Luna. A $G$-homogeneous variety $Y$ is said to be \emph{induced} from a parabolic $P$ if it is of the form $Y'\times^P G$, where $Y'$ is a homogeneous spherical variety for the Levi quotient of $P$; equivalently, $Y=H\backslash G$, where $H\subset P$ contains the unipotent radical of $P$.

\begin{proposition}\cite[Proposition 3.4]{Lu}\label{parindcriterion}
 A homogeneous spherical variety $Y$ for $G$ is induced from a parabolic $\bar P$ (assumed opposite to a standard parabolic $P$) if and only if the union of $\Delta(Y)$ with the support\footnote{The support of a subset in the span of $\Delta$ is the smallest set of elements of $\Delta$ in the span of which it lies.} of the spherical roots of $Y$ is contained in the set of simple roots of the Levi subgroup of $P$.
\end{proposition}

\begin{proof}[Proof of Proposition \ref{parind}]
 For every $G$-orbit $Y$ in a spherical variety $X$ there is a simple toroidal variety $\tilde X$ with a morphism $\tilde X\to X$ which is birational and whose image contains $Y$. Therefore, it suffices to assume that $X$ is a simple toroidal variety.

Moreover, if $\bar X$ denotes the wonderful compactification of $X^+$ (i.e.\ the simple toroidal compactification with $\mathcal C(\bar X)=\mathcal V$) then every simple toroidal variety $X$ admits a morphism $X\to \bar X$ which, again, is birational and has the property that every non-open $G$-orbit on $X$ goes to a non-open $G$-orbit in $\bar X$. Indeed, any non-open $G$-orbit $Y\subset X$ corresponds to a non-trivial face of $\mathcal C(X)$, and its character group $\varchi(Y)$ is the orthogonal complement of that face in $\varchi(X)$, which is of lower rank than $\varchi(X)$, therefore $Y$ has to map to an orbit of lower rank. 
Moreover, $Y$ is a torus bundle over its image. This reduces the problem to the case where $X$ is a wonderful variety, which we will now assume.

By Proposition \ref{parindcriterion}, it suffices to show that the union of $\Delta(X)$ and the support of the spherical roots of $Y$ is not the whole set $\Delta$ of simple roots. The spherical roots of $Y$ are a proper subset of the spherical roots of $X$, and $\Delta(Y)=\Delta(X)$. It therefore suffices to prove that for any proper subset $\Theta\subset\Delta_X$ there exists a simple root $\alpha\in \Delta\smallsetminus \Delta(X)$ such that $\alpha$ is not contained in the support of $\Theta$.

Denote $\mathfrak a^*:=\varchi(A)^*\otimes \QQ$, $\mathfrak a^*_{P(X)}=(\Delta(X))^\perp\subset\mathfrak a^*$, and consider the canonical quotient map: $q:\mathfrak a \to \mathcal Q$. Denote by $\mathfrak f_\emptyset \subset \mathfrak a^*$ the anti-dominant Weyl chamber in $\mathfrak a$. Every set of spherical roots $s\subset\Delta_X$ corresponds to a face $\mathcal V_s\subset\mathcal V=\mathcal V_\emptyset\subset \mathcal Q$ (more precisely, $\mathcal V_s$ is the face spanning the orthogonal complement of $s$), and similarly every set $r\subset \Delta$ of simple roots of $G$ corresponds to a face $\mathfrak f_r\subset \mathfrak f_\emptyset$. The simple roots in the support of $\gamma\in \Delta_X$ are those corresponding to the largest face $\mathfrak f$ of $\mathfrak f_\emptyset$ which is contained in $q^{-1}(\mathcal V_{\{\gamma\}})$. Notice that the maximal vector subspace $\mathfrak f_\Delta$ of $\mathfrak f_\emptyset$ maps into the maximal vector subspace $\mathcal V_{\Delta_X}$ of $\mathcal V$.

By assumption, $\mathfrak f_\emptyset$ surjects onto $\mathcal V$. Moreoever, since every element of $\mathfrak f_\emptyset$ can be written as a sum of an element in $\mathfrak f_{\Delta(X)}$ and a non-negative linear combination of $\check \Delta(X):=\{\check\alpha|\alpha\in\Delta(X)\}$, and since $\check\Delta(X)$ is in the kernel of $\mathfrak a\to\mathcal Q$, it follows that $\mathfrak f_{\Delta(X)}$ surjects onto $\mathcal V$. Now let $\Theta\subset\Delta_X$ be a proper subset. Let $\mathfrak f_s$ be a face of $\mathfrak f_{\Delta(X)}$ which surjects onto $\mathcal V_\Theta$. Since $\mathfrak f_s\neq \mathfrak f_\Delta$, there is an $\alpha\in \Delta\smallsetminus \Delta(X)$ which is not in the support of $\Theta$. 
\end{proof}

\subsection{Tensor product $L$-functions of $GL_2$ cusp forms} \label{sstensor}

In section \ref{secSchwartz} we proposed a general conjecture involving distributions which are obtained from the geometry of an affine spherical variety $X$, and in this section we saw how this conjecture is true, and gives rise to period- and Rankin-Selberg integrals, in the case that $X$ admits the structure of a ``pre-flag bundle''. It was written above that such a structure should be considered essentially irrelevant and a matter of ``chance''. We now wish to provide some evidence for this point of view by recalling the known constructions of $n$-fold tensor product $L$-functions for $\GL_2$, where $n\le 3$. The point is that while these constructions seem comletely different from the point of view of Rankin-Selberg integrals, from the point of view of spherical varieties they are completely analogous!

\label{ssrelatedtolfunction}Before we consider the specific example, let us become a bit more precise about what it means that \emph{a period integral is related to some $L$-value}. Let $\pi=\otimes' \pi_v$  be an (abstract) unitary representation of $G(\Ad)$, the tensor product of unitary irreducible representations $\pi_v$ of $G(k_v)$ with respect to distinguished unramified vectors $u_v^0$ (for almost every place $v$) of norm 1.  Let $\mathcal P$ be a functional on $\pi$. In our applications the functional $\mathcal P$ will arise as the composition of a cuspidal automorphic embedding $\nu:\pi \to L^2_\cusp(G(k)\backslash G(\Ad))$, assumed unitary, with a period integral or, more generally, the pairing (\ref{intconj1}) with a fixed $X$-theta series. Let $\rho$ be a representation of the dual group, and let $L(\pi,\rho,s)$ denote the value of the corresponding $L$-function at the point $s$. We say that $|\mathcal P|^2$ is \emph{related to $L(\pi,\rho,s)$} if there exist non-zero skew-symmetric forms: $\
Lambda_v:\pi_v\otimes \bar \pi_v\to \CC$ for every $v$ such that for any large enough set of places $S$, and for a vector $u=\otimes_{v\in S} u_v^0 \otimes_{v\notin S}u_v$ one has: $|\mathcal P(u)|^2= L^S(\pi,\rho,s) \cdot \prod_{v\in S} \Lambda_v (u_v, \overline u_v).$ (Of course, for this to happen we must have $\Lambda_v(u_v^0,\overline u_v^0)=L_v(\pi_v,\rho_v, s)$.)
Moreover, it is required that each $\Lambda_v$ has a definition which has no reference to any other representation but $\pi_v$. The reader will notice that the last condition does not stand the test of mathematical rigor; however, not imposing it would make the rest of the statement void up to whether $\mathcal P$ is zero or not. In practice, the $\Lambda_v$'s will be given by reference to some non-arithmetic model for $\pi_v$. See \cite{II} for a precise conjecture in a specific case, and \cite{SV} for a more general but less precise conjecture.\footnote{For the sake of completeness, we should mention that when $\mathcal P$ comes from a period integral one should in general modify the above conjecture by some ``mild'' arithmetic factors, such as the sizes of centralizers of Langlands parameters -- see \cite{II}. However, in the example we are about to discuss there is no such issue since the group is $\GL_2$.}

\begin{example}
 If $\mathcal P$ denotes the Whittaker period: $$\phi\mapsto \int_{U(k)\backslash U(\Ad)} \phi(u) \psi^{-1}(u) du$$ (where $\psi$ is a generic idele class character of the maximal unipotent subgroup) on cusp forms for $G=\GL_n$, then $|\mathcal P|^2$ is related to the $L$-value:
$$ \frac{1}{L(\pi,\operatorname{Ad}, 1)}$$
cf.\ \cite{Ja, SV}. Notice that the examples which we are about to discuss admit ``Whittaker unfolding'' and this factor will enter, although most references introduce a different normalization and ignore this factor.
\end{example}

Now we are ready to discuss our example: Let $n$ be a positive integer, $G=(\GL_2)^n \times \Gm$, and let $H$ be the subgroup:
$$\left\{ \left.\left(\begin{array}{cc} a & x_1\\ & 1\end{array}\right) \times \left(\begin{array}{cc} a & x_2\\ & 1\end{array}\right) \times \cdots \times\left(\begin{array}{cc} a & x_n\\ & 1\end{array}\right) \times a\right| x_1+ x_2 + \dots+ x_n = 0\right\}.$$
We let $X=\overline{H\backslash G}^\aff$. As usual, we normalize the action of $G$ on functions on $X^+$ so that it is unitary with respect to the natural measure. Let us see that for $n=1,2,3$ the variety $X$ admits the structure of a pre-flag bundle, and hence the integral of Conjecture \ref{weakconjecture} can be interpreted as a Rankin-Selberg integral, as discussed above:

\begin{itemize}
 \item $n=1$. Here $\overline{H\backslash G}^\aff=H\backslash G$ and we get the integral (\ref{hecke}) of Hecke. If $\tau_s=\tau\otimes |\bullet|^s$, where $\tau$ is a cuspidal representation of $\GL_2$ (for simplicity: with trivial central character), the square of the absolute value of the corresponding linear functional on $\tau_s\otimes\widetilde{\tau_s}$ is related to the $L$-value:
$$\frac{L(\tau, \frac{1}{2}+s) L(\tilde \tau, \frac{1}{2}-s)}{L(\tau,\operatorname{Ad},1)}.$$
 \item $n=2$. Here the projection of $H$ to $GL_2^2$ is conjugate to the mirabolic subgroup of $GL_2$ embedded diagonally. Therefore, the affine closure of $H\backslash G$ is equal to the bundle over $\GL_2^\diag\backslash (\GL_2)^2$ with fiber equal to the affine closure of $U_2\backslash \GL_2$, where $U_2$ denotes a maximal unipotent subgroup of $\GL_2$. Corresponding to this pre-flag bundle is a Rankin-Selberg integral ``with the Eisenstein series on the smaller group'' $\GL_2^\diag$, namely the classical integral of Rankin and Selberg. If $\tau=\tau_1\otimes\tau_2\otimes |\bullet|^s$ is a cuspidal automorphic representation of $G$ (for simplicity: with trivial central character), the square of the absolute value of the corresponding integral is related to the $L$-value:
$$\frac{L(\tau_1\otimes\tau_2, \frac{1}{2}+s) L(\tilde \tau_1\otimes\tilde\tau_2, \frac{1}{2}-s)}{L(\tau,\operatorname{Ad},1)}.$$
 \item $n=3$. In this case there is a structure of a pre-flag variety not on $X$, but on $X^0$: the corresponding spherical variety for the subgroup $G^0=\{(g_1,g_2,g_3,a)\in G | \det(g_1)=\det(g_2)=\det(g_3)\}$. The structure of a pre-flag variety involves the group $\tilde G =\operatorname{GSp}_6$ and the subgroup $\tilde H=[\tilde P,\tilde P]$, where $\tilde P$ is the Siegel parabolic -- this is a construction of Garrett \cite{GaTr}. The group $(\GL_2^3)^0$ is embedded in $\GSp_6$ as $(\GSp_2^3)^0$. Then, according to \cite[Corollary 1 to Lemma 1.1]{PSRtriple} the group $G^0$ has an open orbit in $[\tilde P,\tilde P]\backslash \tilde G$ with stabilizer equal to $H$. We claim:
\begin{lemma}
 The affine closure $X^0$ of $H\backslash G^0$ is equal to the affine closure of $[\tilde P,\tilde P]\backslash \tilde G$.
\end{lemma}
\begin{proof}
 Denote by $Y$ the affine closure of $[\tilde P,\tilde P]\backslash \tilde G$. We have an open embedding: $X^0\hookrightarrow Y$. By \cite[Lemma 1.1]{PSRtriple}, all non-open $G$-orbits have codimension at least two. Therefore, the embedding is an isomorphism.
\end{proof}
Hence, our integral for $X^0$ coincides with the Rankin-Selberg integral of Garrett. The only thing that remains to do is to compare the normalizations for the sections of Eisenstein series. From \cite[Theorem 3.1]{PSRtriple} one sees that the square of the absolute value of our integral is related to the $L$-value:
$$\frac{L(\tau_1\otimes\tau_2\otimes \tau_3, \frac{1}{2}+s) L(\tilde \tau_1\otimes\tilde\tau_2\otimes\tilde\tau_3, \frac{1}{2}-s)}{L(\tau,\operatorname{Ad},1)}.$$
(Again, for simplicity, we assume trivial central characters. Notice that the zeta factors in \cite[Theorem 3.1]{PSRtriple} disappear because of the correct normalization of the Eisenstein series!)
\end{itemize}

It is completely natural to expect the corresponding integral for $n=4$ or higher to be related to the $n$-fold tensor product $L$-function. It becomes obvious from the above example that the point of view of the spherical variety is the natural setting for such integrals, while at the same time the structure of a pre-flag bundle may not exist and, even if it exists, it has a completely different form in each case which conceals the uniformity of the construction.


\section{Smooth affine spherical varieties}\label{secsmoothaffine}

Given that we do not know how to prove Conjecture \ref{weakconjecture}, except in the cases of wavefront pre-flag bundles, it is natural to ask the purely algebro-geometric question: Which spherical varieties admit the structure of a pre-flag bundle? An answer to this question would amount to a complete classification of Rankin-Selberg integrals, in the restricted sense that ``Rankin-Selberg'' has been used here. Such an answer has been given in the special case of \emph{smooth} affine spherical varieties: These varieties automatically have the structure of a pre-flag bundle, and they have been classified by Knop and Van Steirteghem \cite{KnVS}, hence can be used to produce Eulerian integrals of automorphic forms! There seems to be little point in computing every single example in the tables of \cite{KnVS}, and my examination of most of the cases has not produced any striking new examples. However, we get some of the best-known integral constructions, as well as some new ones (which do not produce any 
interesting new $L$-functions).

\subsection{Smooth affine spherical triples}  By Theorem \ref{Lunacorollary} of Luna, every smooth affine spherical variety of $G$ (over an algebraically closed field in characteristic zero) is of the form $V\times^H G$, where $H$ is a reductive subgroup (so that $H\backslash G$ is affine) and $V$ is an $H$-module. As we have seen in Example \ref{vsexample}, vector spaces are pre-flag varieties, and therefore all smooth affine spherical varieties are pre-flag bundles. We check the details carefully:

\begin{lemma}
 Every smooth affine spherical variety admits the structure of a pre-flag bundle.\footnote{Strictly speaking, the ``affine closure'' condition is not satisfied when the fibers have one-dimensional summands under the action of $\mathcal Z(X)$; one should modify the definition of a pre-flag bundle to allow this case, but in order not to overcomplicate things we prefer not to do so. Notice that after integrating by characters of $\mathcal Z(X)$ the ``basic function'' of  $\Gm$ differs from the ``basic function'' of $\Ga$ only by a Dirichlet $L$-function, so the meromorphic properties of the integrals we are considering are not affected by whether we compactify $\Gm$ or not.}
\end{lemma}
\begin{proof}
 If $X= V\times^H G$ as above, we set $Y = (\mathcal N(H)^0\cdot H)\backslash G$. We let $\tilde X^+$ be the subvariety on which $\mathcal Z(X)$ acts freely, and take $\tilde G = G$. Clearly, $\mathcal Z(X)$ contains the connected centralizer of $H$ in $\GL(V)$ (which is a torus, since $X$ is spherical), so if $V=\oplus_i V_i$ is the decomposition into irreducible $H$-representations according to $\mathcal Z(H)^0$ then $\tilde X^+ = \prod_i (V_i\smallsetminus\{0\}) \times^H G$, and $G$ acts transitively on $\tilde X^+$. By the assumption $\mathcal Z(X)=\mathcal Z(G)^0$, $\mathcal Z(X)$ is the connected center of $\mathcal N(H)$, hence $\tilde Y:=\tilde X^+/\mathcal Z(X)$ has fibers $\mathbb PV_1\times \dots \times \mathbb PV_n$ over $Y$ and is therefore proper over $Y$.
\end{proof}

The corresponding integrals include all period integrals over reductive subgroups, as well as Rankin-Selberg integrals involving \emph{mirabolic} Eisenstein series (i.e.\ those induced from the mirabolic subgroup of $\GL_n$).

In \cite{KnVS}, Knop and Van Steirteghem classify all smooth affine spherical \emph{triples} $(\gf, \hh, V)$, which amounts to a classification of smooth affine spherical varieties up to coverings, central tori and $\Gm$-fibrations. We recall their definitions:

\begin{definition}
\begin{enumerate}
\item Let $\hh\subset\gf$ be semisimple Lie algebras and let $V$ be a representation of $\hh$. For $\ss$, a Cartan subalgebra of the centralizer $c_\gf (\hh)$ of $\hh$, put $\bar\hh := \hh \oplus \ss$, a maximal central extension of $\hh$ in $\gf$. Let $\zz$ be a Cartan subalgebra of $\mathfrak{gl}(V)^\hh$ (the centralizer of $\hh$ in $\mathfrak{gl}(V)$). We call $(\gf, \hh, V)$ a \emph{spherical triple} if there exists a Borel subalgebra $\bb$ of $\gf$ and a vector $v\in V$ such that 
\begin{enumerate}
\item 
$\bb + \bar\hh = \gf$ and
\item $[(\bb \cap \bar\hh) + \zz]v = V$ where $\ss$ acts via any homomorphism $\ss\to\zz$ on $V$.
\end{enumerate}
\item Two triples $(\gf_i , \hh_i , V_i)$, $i = 1, 2$, are \emph{isomorphic} if there exist isomorphisms of Lie algebras resp.\ vector spaces $\alpha: \gf_1\to\gf_2$ and $\beta: V_1\to V_2$ such that 
\begin{enumerate}
\item
$\alpha(\hh_1 ) = \hh_2$ 
\item
$\beta(\xi v) = \alpha(\xi)\beta(v)$ for all $\xi\in \hh_1$ and $v\in V_1$.
\end{enumerate}
\item 
Triples of the form $(\gf_1\oplus\gf_2, \hh_1\oplus\hh_2, V_1\oplus V_2)$ with $(g_i , h_i , V_i ) \ne (0, 0, 0)$ are called \emph{decomposable}. 
\item 
Triples of the form $(\kk, \kk, 0)$ and $(0, 0, V )$ are said to be \emph{trivial}. A pair $(\gf, \hh)$ of semisimple Lie algebras is called spherical if $(\gf, \hh, 0)$ is a spherical triple. 
\item A spherical triple (or pair) is \emph{primitive} if it is non-trivial and indecomposable.
\end{enumerate}
\end{definition}

Clearly, every smooth affine spherical variety gives rise to a spherical triple. Conversely, each spherical triple is obtained from a (not necessarily unique) smooth affine spherical variety, as follows by an a posteriori inspection of all spherical triples. (The non-obvious step here is that the $\hh$-module $V$ integrates to an $H$-module, where $H$ is the corresponding subgroup.) 

The classification of all primitive spherical triples is given in \cite{KnVS}, Tables 1, 2, 4 and 5, modulo the inference rules described in Table 3. The diagrams are read in the following way: The nodes in the first row correspond to the simple direct summands $\gf_i$ of $\gf$, the ones in the second row to the simple direct summands $\hh_i$ of $\hh$ and the  ones in the third row to the simple direct summands $V_i$ of $V$. If $(\gf,\hh)$ contains a direct summand of the form $(\hh_1,\hh_1)$ then the $\hh_1$ summand is omitted from the first row There is an edge between $\gf_i$ and $\hh_j$ if $\hh_j\hookrightarrow \gf\twoheadrightarrow\gf_i$ is non-zero and an edge between $h_j$ and $V_k$ if $V_k$ is a non-trivial $h_j$-module. The edges are labeled to describe the inclusion of $\hh$ in $\gf$, resp.\ the action of $\hh$ on $V$; the labels are omitted when those are the ``natural'' ones. 

We number the cases appearing in the list of Knop and Van Steirteghem as follows: First, according to the table on which they appear (Tables 1, 2, 4, 5 in \cite{KnVS}); and for each table, numbered from left to right, top to bottom. 

\subsection{Eulerian integrals arising from smooth affine varieties}  

In what follows we will discuss a sample of the global integrals obtained from varieties in the list of Knop and Van Steirteghem. At this point it is more convenient not to normalize the action of $G$ unitarily. We allow ourselves to choose the spherical variety corresponding to a given spherical triple as is most convenient, and in fact we sometimes replace semisimple groups by reductive ones. Of course, the classification in \cite{KnVS} is over an algebraically closed field, which leaves a lot of freedom for choosing the precise form of the spherical variety over $k$, even when $G$ is split. In the discussion which follows we will always take both the group and generic stabilizer to be split. Many of the varieties of Knop and Van Steirteghem have zero cuspidal contribution (i.e.\ the integral (\ref{intconj1}) is zero for every cusp form) or are not multiplicity-free. Still, this list contains some of the best-known examples of integral representations of $L$-functions. It contains also some new ones.

In subsection \S\ref{ssrelatedtolfunction} we explained what it means for a period integral $\mathcal P$ to be ``related to'' an $L$-value, namely by considering the value of $\mathcal P|_\pi\cdot \mathcal P|_{\bar \pi}$, assuming that $\pi$ is an abstract unitary representation of an adelic group, embedded unitarily into the space of cuspidal automorphic forms for that group. For the examples that we are about to see, we will adopt a language that describes the value of $\mathcal P|_\pi$ itself, divided by the value of a period integral that does not depend on a continuous parameter, such as the Whittaker period. For example, for the Hecke integral (\ref{hecke}) we would say that it is related to $L(\pi,s+\frac{1}{2})$ with respect to Whittaker normalization, while for the Godement-Jacquet integral (\ref{GodementJacquet}) we would say that it is related to $L(\pi, s-\frac{1}{2}(n-1))$ with respect to the ``inner product'' period on $\pi\otimes\tilde\pi$.

\subsubsection{\textsc{Table 1}}

In this table the group $H$ is equal to $G$, i.e. the data consists of a group and a spherical representation of it. This table contains the following interesting integrals (numbered according to their occurence in the tables of Knop and Van Steirteghem):

\begin{description}
 \item[1. \textbf{The integrals of Godement and Jacquet}.]
 Here the group is $\GL_n\times \GL_m$ with the tensor product representation (i.e.\ on $\Mat_{n\times m}$). It is easy to see that if $m\ne n$ then the stabilizer is parabolically induced, hence the only interesting case (as far as cusp forms are concerned) is $m=n$. In this case, our integral (\ref{intconj1}) \emph{is that of Godement and Jacquet}:
$$ \int_{Z^\diag(\Ak)\GL_n^\diag(k)\backslash \GL_n(\Ak)\times\GL_n(\Ak)} \phi_1(g_1)\phi_2(g_2) \Phi(g_1^{-1}g_2) \cdot $$ $$\cdot |\det(g_1^{-1}g_2)|^{s} d(g_1,g_2).$$


 \item[15. \textbf{Two new integrals.}] (Here there is a choice between the first and the last fundamental representation of $\GL_n$. It can easily be seen that they amount to the same integral, so we will consider only $\omega_1$.)

The group is $\GL_m\times \GL_n$ and the representation is the direct sum $\Mat_{m\times n}$ with the standard representation for $\GL_n$. If $m\ne n, n-1$ then we can easily see that the stabilizer is parabolically induced. Hence there are two interesting cases:

\begin{enumerate}
\item[(i)] $m=n$. We let $\phi_1 \in \pi_1,\phi_2\in \pi_2$ be two cusp forms on $\GL_n$. Then the integral is:
$$ \int_{P_n^\diag(k)\backslash \GL_n(\Ak)\times\GL_n(\Ak)} \phi_1(g_1)\phi_2(g_2) \Phi(g_1^{-1}g_2) \Phi'([0, \dots, 0, 1]\cdot g_1)\cdot $$ \nopagebreak $$\cdot |\det(g_1^{-1}g_2)|^{s_1} |\det(g_1)|^{s_2} dg_1 dg_2.$$
Here $\Phi$ is a Schwartz function on $\Mat_n(\Ak)$ and $\Phi'$ is a Schwartz function on $\Ak^n$.

\begin{theorem}
The above integral is Eulerian and with respect to Whittaker normalization is related to the $L$-value:
\begin{equation}L(\pi_1\otimes\pi_2, s_2)\cdot L(\pi_2, s_1-\frac{1}{2}(n-1)).
\end{equation}
\end{theorem}

\begin{proof}
It follows from the standard ``unfolding'' technique that the above integral, in the domain of convergence, is equal to:
$$\int_{(U_n(\Ad)\backslash \GL_n(\Ad))^2} W_1 (g_1) W'_2(g_2) \Phi(g_1^{-1}g_2) \Phi'([0, \dots, 0, 1]\cdot g_1)\cdot $$ \nopagebreak $$\cdot |\det(g_1^{-1}g_2)|^{s_1} |\det(g_1)|^{s_2}  dg_1 dg_2$$
where $W_1(g)=\int_{U_n(k)\backslash U_n(\Ad)} \phi_1(ug) \psi(u) du$
and $W'_2$ the same with $\phi_1$ replaced by $\phi_2$ and $\psi$ replaced by $\psi^{-1}$. 

The last integral is (for ``factorizable data'') a product of local factors:
$$\int_{(U_n(k_v)\backslash \GL_n(k_v))^2} W_{1,v} (g_1) W'_{2,v}(g_2) \Phi_v(g_1^{-1}g_2) \Phi_v'([0, \dots, 0, 1]\cdot g_1)\cdot $$ \nopagebreak $$\cdot |\det(g_1^{-1}g_2)|^{s_1} |\det(g_1)|^{s_2}  dg_1 dg_2.$$

Assume that $\Phi_v=\Phi_v^0$, the basic function of $\mathcal S(\Mat_n(k_v))$. Considering the action of the spherical Hecke algebra of $G_2$ (=the second copy of $\GL_n$) on $\mathcal S(\Mat_n(k_v))$, the work of Godement and Jacquet \cite[Lemma 6.10]{GJ} proves:
\begin{equation}
 \Phi_v^0(x) |\det(x)|^{s_1} = \widecheck\Sat_{G_2} \left(\frac{1}{\wedge^\top \left(1- q_v^{-s_1+\frac{1}{2}(n-1)}\cdot\operatorname{std}\right)}\right) \star 1_{\GL_n(\mathfrak o)}
\end{equation}
Therefore for unramified data the last integral is equal to:

$$L(\pi_2, s_1-\frac{1}{2}(n-1)) \cdot \int_{(U_n(k_v)\backslash \GL_n(k_v))^2} W_{1,v} (g_1) W'_{2,v}(g_2) \cdot $$ $$\cdot1_{\GL_n(\mathfrak o_v)}(g_1^{-1}g_2) \Phi_v'([0, \dots, 0, 1]\cdot g_1) |\det(g_1^{-1}g_2)|^{s_1} |\det(g_1)|^{s_2}  dg_1 dg_2 =$$
$$= L(\pi_2, s_1-\frac{1}{2}(n-1)) \cdot $$ $$\cdot\int_{(U_n(k_v)\backslash \GL_n(k_v))} W_{1,v} (g) W'_{2,v}(g)  \Phi_v'([0, \dots, 0, 1]\cdot g) |\det(g)|^{s_2}  dg.$$

The latter is the classical Rankin-Selberg integral, which with respect to Whittaker normalization is related to $L(\pi_1\otimes\pi_2, s_2)$ (see, for instance, \cite{Co}).
\end{proof}

\item[(ii)] $m=n-1$. Notice that if $V$ denotes the standard representation of $\GL_n$ then the space $\Mat_{(n-1)\times n}\oplus V$ can be identified under the $G_1\times G_2:=\GL_{n-1}\times \GL_n$-action with the space $X=\Mat_n$, where $g\in G_1$ acts as $\left(\begin{array}{cc} g^{-1} \\ & 1\end{array}\right)$ on the left Let $\phi_1\in \pi_1$ be a cusp form on $\GL_{n-1}$ and $\phi_2\in \pi_2$ a cusp form in $\GL_n$. Then the integral is:
$$ \int_{\GL_n^\diag(k)\backslash \GL_{n+1}(\Ak)\times\GL_n(\Ak)} \phi_1(g_1)\phi_2(g_2) \cdot $$ $$ \cdot \Phi\left(\left(\begin{array}{cc} g_1^{-1} \\ & 1\end{array}\right) g_2\right) \left|\frac{\det(g_2)}{\det(g_1)}\right|^{s_1} |\det(g_1)|^{s_2} dg_1 dg_2$$
where $\Phi\in\mathcal S(\Mat_n(\Ad))$.

As before, one can prove:
\begin{theorem}
The above integral is Eulerian and with respect to Whittaker normalization related to the $L$-value:
\begin{equation}L(\pi_1\otimes\pi_2, s_2+\frac{1}{2})\cdot L(\pi_2, s_1-\frac{1}{2}n).
\end{equation}
\end{theorem}

\end{enumerate}

\end{description}

\subsubsection{\textsc{Table 2}}

In this table $H$ is smaller than $G$ and the representation $V$ of $H$ is non-trivial. This table contains the following interesting integrals:

\begin{description}
\item[1. \textbf{The Bump-Friedberg integral.}] The group is $\GL_{m+n}$ where $m=n\text{ or }n+1$, the subgroup $H$ is $\GL_m\times \GL_n$ and the representation is the standard representation of the second factor. This is the integral examined in \cite{BF}:
$$ \int_{\GL_m(k)\times \GL_n(k)\backslash \GL_m(\Ak)\times\GL_n(\Ak)} \phi\left(\begin{array}{cc} g_1 \\ & g_2\end{array}\right) \left|\frac{\det (g_1)}{\det(g_2)}\right|^{s_1} \cdot $$ $$ \cdot \Phi([0, \cdots, 0, 1] \cdot g_2) |\det g_2|^{s_2} dg_1 dg_2$$
 It is related with respect to Whittaker normalization to the $L$-value:
$$L(\pi,s_1+\frac{1}{2})L(\pi,\wedge^2, s_2).$$

\item[3. \textbf{A new integral.}] The group is $\GL_{m+1}\times \GL_n$, and $G'=\GL_m\times \GL_n$ with the tensor product of the standard representations (i.e.\ on $\Mat_{m\times n}$). The only interesting case is $m=n$. If $n>m$ then the stabilizer is parabolically induced, and when $m>n$ it unfolds to a parabolically induced model.

If $m=n$ we get:
$$ \int_{\GL^\diag(k)\backslash \GL_n(\Ak)\times\GL_n(\Ak)} \phi_1\left(\begin{array}{cc} g_1 \\ & 1\end{array}\right)\phi_2(g_2) \Phi(g_1^{-1} g_2) \cdot $$ $$ \cdot \left|\frac{\det(g_2)}{\det(g_1)}\right|^{s_1} |\det(g_1)|^{s_2} d(g_1,g_2).$$

As before, one can prove:

\begin{theorem}
The above integral is Eulerian and with respect to Whittaker normalization related to the $L$-value:
\begin{equation}L(\pi_1\otimes\pi_2, s_2+\frac{1}{2})\cdot L(\pi_2, s_1-\frac{1}{2}(n-1)).
\end{equation}
\end{theorem}

\item[5. \textbf{The classical Rankin-Selberg integral.}] The group is $\GL_n\times \GL_n$ and the subgroup $G'$ is $\GL_n^\diag$ with the standard representation. This is the classical Rankin-Selberg integral:
$$ \int_{\GL_n(k)\backslash\GL_n(\Ak)} \phi_1(g)\phi_2(g) \Phi([0, \cdots , 0, 1]\cdot g) |\det g|^s dg.$$

It is related with respect to Whittaker normalization to the $L$-value (cf.\ \cite{Co}):
$$L(\pi_1\otimes\pi_2, s).$$

\end{description}












\subsubsection{\textsc{Tables 4 and 5}}

Here the representation $V$ is trivial, hence we get period integrals over reductive algebraic subgroups (\S \ref{ssperiods}). All known cases of multiplicity-free period integrals are contained in these tables.


 




\section{A remark on a relative trace formula}\label{secRTF}

At this point we drop our assumptions on the group $G$, in order to discuss non-split examples. We will assume the existence of Schwartz spaces with similar properties in this setting, in order to give a conceptual explanation to the phenomenon of ``weight factors'' in a relative trace formula.

The relative trace formula is a method which was devised by Jacquet and his coauthors to study period integrals of automorphic forms. In its most simplistic form, it can be described as follows: Let $H_1$ and $H_2$ be two reductive spherical subgroups of $G$ (a reductive group defined over a global field $k$) and let $f\in C_c^\infty(G(\Ad))$. Then one builds the usual kernel function: $K_f(x,y)=\sum_{\gamma\in G(k)} f(x^{-1}\gamma y)$ for the action of $f$ on the space of automorphic functions and (ignoring analytic difficulties) defines the functional:
\begin{equation}\label{simpleRTF}
\RTF_{H_1,H_2}^G(f) = \int_{H_1(k)\backslash H_1(\Ad)} \int_{H_2(k)\backslash H_2(\Ad)} K_f(h_1, h_2) dh_1 dh_2.
\end{equation}
The functional can be decomposed in two ways, one geometric and one spectral, and the spectral expansion involves period integrals of automorphic forms. By comparing two such RTFs (i.e.\ made with different choices of $H_1, H_2$, maybe even different groups $G$) one can deduce properties of those period integrals, such as that their non-vanishing characterizes certain functorial lifts.

The above presentation is too simplistic for several reasons: First, the correct functional has something to do with the stack-theoretic quotient $H_1\backslash G/H_2$, which sometimes forces one to take a sum over certain inner forms of $G$ and $H_i$. We will not discuss stack-theoretic quotients or inner forms here, but at first approximation we observe that from this algebro-geometric point of view the variety $H_i\backslash G$ is more natural than the space $H_i(k)\backslash G(k)$; hence, if $G(k)$ does not surject onto $(H_i\backslash G_i)(k)$ one should take the sum of the above expressions over stabilizers $H_{i,\epsilon}$ of a set of representatives of $G(k)$-orbits. (This will become clearer in a reformulation which we will present below.) Moreover, one can consider an idele class character $\eta$ of $H_i$ and integrate against this character; we will adjust our notation accordingly, for instance: $\RTF^G_{H_1,(H_2,\eta)}$. There are often analytic difficulties in making sense of the above integrals.
 And one does not have to restrict to reductive subgroups, but can consider parabolically induced subgroups together with a character on their unipotent radical (such as in the Whittaker period). However, we will ignore most of these issues and focus on another one, first noticed in \cite{JLR}: It seems that in certain cases, in order for the relative trace formula $\RTF^G_{H_1,H_2}$ to be comparable to some other relative trace formula, the functional (\ref{simpleRTF}) is not the correct one and one has to add a ``weight factor'' in the definition, such as:
\begin{equation}\label{thetaRTF}
\RTF_{H_1,H_2}^G(f) = \int_{H_1(k)\backslash H_1(\Ad)} \int_{H_2(k)\backslash H_2(\Ad)} K_f(h_1, h_2) \theta(h_1) dh_1 dh_2
\end{equation}
where $\theta$ is a suitable automorphic form on $H_1$.

Our goal here is to explain how, under the point of view developed in the present paper, the above expression is not a relative trace formula for $H_1, H_2$ but represents a relative trace formula for some \emph{other} subgroups. We will discuss this in the context of \cite{JLR}, though our starting point will not be (\ref{thetaRTF}) but another formula of \cite{JLR} from which the identities for (\ref{thetaRTF}) are derived, and which is closer to our point of view.

More precisely, let $E/F$ be a quadratic extension of number fields with corresponding idele class character $\eta$, $G= \Res_{E/F}\PGL_2$, $G'=\PGL_2\times \PGL_2$ (over $F$), $H\subset G$ the projectivization of the quasi-split unitary group (which is in fact split, i.e.\ isomorphic to $PGL_2$ over $F$), $H'=$ the diagonal copy of $\PGL_2$ in $G'$. (Compared to \cite{JLR}, we restrict to $\PGL_2$ for simplicity.) We consider $\eta$ as a character of $H$ in the natural way. Naively, one would like to compare the functional: $\RTF_{H,(H,\eta)}^G$ to the functional $\RTF_{H',H'}^{G'}$ (usual trace formula for $G'$). However, it turns out that the correct comparison is between the functionals:
\begin{equation}\label{1stmod}
 f\mapsto \int_{(H(k)\backslash H(\Ad))^2} K_f(h_1, h_2) E(h_1,s) \eta(h_1) dh_1 dh_2
\end{equation}
on $G$ and 
\begin{equation}\label{2ndmod}
 f'\mapsto \int_{(H'(k)\backslash H'(\Ad))^2} K_{f'}(h'_1, h'_2) E'(h'_1,s) dh'_1 dh'_2
\end{equation}
on $G'$, where $E, E'$ are suitable Eisenstein series on $H, H'$. (More precisely, in the first case one takes the sum over the unitary groups of all $G(k)$-conjugacy classes of non-degenerate hermitian forms for $E/F$, as we mentioned above, \emph{but only in the second variable}.)

Notice that we have already made a modification to the formulation of \cite{JLR}, namely in the second case they let $G'=\PGL_2$ and consider the integral: $$\int_{\PGL_2(k)\backslash \PGL_2(\Ad)} K_{f'}(x,x) E'(x,s) dx,$$ but this is easily seen to be equivalent to our present formulation.

\begin{claim}
 The functionals (\ref{1stmod}), (\ref{2ndmod}) can naturally be understood as pairings:
$$\RTF_{X_1,X_2}^{\Gm\times G, \omega}: \mathcal S(X_1(\Ad))\otimes \mathcal S(X_2(\Ad)) \to \CC$$
respectively:
$$\RTF_{X_1',X_2'}^{\Gm\times G',\omega'}: \mathcal S(X_1'(\Ad))\otimes \mathcal S(X_2'(\Ad)) \to \CC$$
where: $X_2=H\backslash G$, $X_2'=H'\backslash G'$ and $X_1, X_1'$ are the affine closures of the varieties:
$$ U_F\backslash G$$
respectively:
$$ U_F' \backslash G'$$
where $U_F, U_F'$ are maximal unipotent subgroups of $H$ resp.\ $H'$.
\end{claim}
 The varieties $X_1$, $X_1'$ are considered here as spherical varieties under $\Gm\times G$ (resp.\ $\Gm\times G'$), where $\Gm=B_2/U_2$, and we extend the $\Gm$-action to the varieties $X_2,X_2'$ in the trivial way. The exponent $\omega$ in $\RTF_{X_1,X_2}^{\Gm\times G, \omega}$ will be explained below.

Before we explain the claim, let us go back to the simpler formula (\ref{simpleRTF}) and explain how it can be considered as a pairing between $\mathcal S(X_1(\Ad))$ and $\mathcal S(X_2(\Ad))$ (where $X_i=H_i\backslash G_i$). Here we will identify Hecke algebras with spaces of functions, by choosing Haar measures. Assume that $f=\check f_1 \star f_2$ with $f_i\in C_c^\infty(G(\Ad))$. Then we set: $\Phi_i(g)=\int_{H_i(\Ad)} f_i(hg)dh$. By the definition of $\mathcal S(X_i(\Ad))$ when $H_i$ is reductive, it follows that $\Phi_i\in\mathcal S(X_i(\Ad))$. (It is at this point that one should add over representatives for $G_i(k)$-orbits on $X_i(k)$, since in general the map $C_c^\infty(G(\Ad))\to \mathcal S(X_i(\Ad))$ is not surjective.) The functional $\RTF^G_{H_1,H_2}(f_1 \star f_2)$ clearly does not depend on $f_1,f_2$ but only on $\Phi_1,\Phi_2$. Hence, it defines a $G^\diag$-invariant functional:
$$ \mathcal S(X_1(\Ad))\otimes \mathcal S(X_2(\Ad)) \to \CC.$$

Now let us return to the setting of the Claim, and of equations (\ref{1stmod}), (\ref{2ndmod}). The product $E(h_1,s)\eta(h_1)$ in (\ref{1stmod}) will be considered as an Eisenstein series on $H(k)\backslash H(\Ad)$. We have seen that suitable sections of Eisenstein series can be obtained from integrating $X$-theta series $\theta_{U_2}^{\Gm\times H}(\Phi,g)$ where $\Phi\in \mathcal S(U_2\backslash H(\Ad))$ against a character $\omega$ of $\Gm$. Now consider $\Phi\in\mathcal S(U_2\backslash H(\Ad))$ as a generalized function on $U_2\backslash G(\Ad)$. Assume again that $f=\check f_1\star f_2\in C_c^\infty(G(\Ad))$. Then $\Phi_1:=f_1\star \Phi \in \mathcal S(U_2\backslash H(\Ad))$ and $\Phi_2(g):=\int_{H_2(\Ad)} f(hg)dg\in \mathcal S(H\backslash G(\Ad))$. Again, of course, we must take many $f$'s and sum over representatives for orbits of $G(k)$ on $X_2(k)$ -- incidentally, our point of view explains why there is no need to sum over representatives for orbits in the first variable: because $G(k)$ surjects on 
$X_1(k)$!

Similarly, one can explain (\ref{2ndmod}) as a pairing between $\mathcal S(X_1'(\Ad))\otimes \mathcal S(X_2'(\Ad))$, and this completes the explanation of our Claim. (We have introduced the exponents $\omega, \omega'$ in the notation, because we have already integrated against the corresponding character of $\Gm$ in order to form Eisenstein series.) Hence, by viewing the Jacquet-Lai-Rallis trace formulae as being attached to the spaces $X_1,X_2$ and $X_1',X_2'$ instead of the original $H\backslash G$ and $H'\backslash G'$, the weight factors do not appear as corrections any more, but as a natural part of the setup.

Notice that this point of view is very close to the geometric interpretation of the fundamental lemma which led to its proof by Ng\^o \cite{Ngo} in the case of the Arthur-Selberg trace formula. Indeed, by the geometric methods of Ng\^o (cf.\ also \cite{GN}) one naturally gets a hold of the orbital integrals of unramified functions arising from intersection cohomology, not the ``naive'' ones defined as characteristic functions of $G(\mathfrak o_v)$-orbits. I hope that this point of view will lead to a more systematic study of the relative trace formula -- at least by alleviating the impression created by weight factors that it is something ``less canonical'' than the Arthur-Selberg trace formula.


\small\noindent{\sc{Department of Mathematics and Computer Science \\
Rutgers University \\
101 Warren Street, Smith Hall 216 \\
Newark, NJ 07102 
}}

\noindent\emph{E-mail address:} \texttt{sakellar@rutgers.edu}

\end{document}